\documentclass[12pt]{article}
\input epsf.tex
\usepackage{amsmath}
\usepackage{amsthm}
\usepackage{amsfonts}
\usepackage{amssymb}
\usepackage{graphicx}
\usepackage{latexsym}
\usepackage{amsmath,amsthm,amsfonts,amscd,amssymb,comment,eucal,latexsym,mathrsfs}
\usepackage{amsmath,amsthm,amsfonts,amssymb}

\usepackage{epsfig}

\usepackage{amssymb}
\usepackage[all]{xy}
\xyoption{poly}
\usepackage{fancyhdr}
\usepackage{wrapfig}

\theoremstyle{plain}
\newtheorem{thm}{Theorem}[section]
\newtheorem{prop}[thm]{Proposition}
\newtheorem{lem}[thm]{Lemma}
\newtheorem{cor}[thm]{Corollary}
\newtheorem{problem}[thm]{Problem}

\theoremstyle{definition}
\newtheorem{defn}{Definition}
\newtheorem{remark}{Remark}
\newtheorem{example}{Example}
\newtheorem{question}{Question}

\newtheorem{notation}{Notation}

\advance \topmargin by -\headheight
\advance \topmargin by -\headsep
\evensidemargin \oddsidemargin
\def\area{{\textrm{area}}}

\def\C{{\mathbb C}}

\def\cc{{\curvearrowright}}
\def\covrad{{\textrm{covrad}}}

\def\dist{{\textrm{dist}}}

\def\EE{{\mathbb E}}

\def\F{{\mathbb F}}
\def\cF{{\mathcal F}}

\def\Free{\textrm{Free}}

\def\cG{{\mathcal G}}

\def\H{{\mathbb H}}

\def\Isom{{\rm{Isom}}}

\def\Leb{{\textrm{Leb}}}
\def\Law{{\textrm{Law}}}

\def\cM{{\mathcal M}}
\def\M{{\mathbb M}}
\def\MS{{\mathbb{MS}}}
\def\MSF{{\mathbb{MSF}}}

\def\DM{{\mathbb{DM}}}

\def\N{{\mathbb N}}

\def\cO{{\cal O}}

\def\Q{{\mathbb{Q}}}

\def\R{{\mathbb R}}

\def\RSC{\textrm{RSC}}

\def\bS{{{S}}}

\def\supp{{\textrm{supp}}}

\def\cT{{\mathcal T}}

\def\chix{{\raise.5ex\hbox{$\chi$}}}

\def\vol{\textrm{vol}}

\def\Z{{\mathbb Z}}

\begin{document}
\title{Cheeger constants and $L^2$-Betti numbers}
\author{Lewis Bowen\footnote{email:lpbowen@math.utexas.edu} \\ University of Texas at Austin}

%\author{University of Hawaii}
%\author[Lewis Bowen]{Lewis Bowen$\dagger$}
%\address{Department of Mathematics\\
%University of Hawai'i--Manoa\\
%} %one \address command per author
%\email{lpbowen@math.hawaii.edu}
%%\thanks{$\dagger$ Supported in part by NSF grants DMS-??.}
%\thanks{$\dagger$ Supported in part by NSF grant DMS-0968762 and NSF CAREER Award DMS-0954606.}
\maketitle
\begin{abstract}
We prove the existence of positive lower bounds on the Cheeger constants of manifolds of the form $X/\Gamma$ where $X$ is a contractible Riemannian manifold and $\Gamma<\Isom(X)$ is a discrete subgroup, typically with infinite co-volume. The existence depends on the $L^2$-Betti numbers of $\Gamma$, its subgroups and of a uniform lattice of $\Isom(X)$. As an application, we show the existence of a uniform positive lower bound on the Cheeger constant of any manifold of the form $\H^4/\Gamma$ where $\H^4$ is real hyperbolic 4-space and $\Gamma<\Isom(\H^4)$ is discrete and isomorphic to a subgroup of the fundamental group of a complete finite-volume hyperbolic 3-manifold. Via Patterson-Sullivan theory, this implies the existence of a uniform positive upper bound on the Hausdorff dimension of the conical limit set of such a $\Gamma$ when $\Gamma$ is geometrically finite. Another application shows the existence of a uniform positive lower bound on the zero-th eigenvalue of the Laplacian of $\H^n/\Gamma$ over all discrete free groups $\Gamma<\Isom(\H^n)$ whenever $n\ge 4$ is even (the bound depends on $n$). This extends results of Phillips-Sarnak and Doyle who obtained such bounds for $n\ge 3$ when $\Gamma$ is a finitely generated Schottky group.
 \end{abstract}

%\thanks{$\dagger$ Supported in part by NSF grant DMS-0968762 and NSF CAREER Award DMS-0954606.}'

\noindent
{\bf Keywords}: Cheeger constant, L2 Betti numbers, hyperbolic manifold\\
{\bf MSC}:57S30 \\

\noindent
\tableofcontents

%\section{Rough Notes}

%1. Covering radius stuff in the proof of the main theorem
%2. Organization of paper must be changed.
%3. Unimodularity and the proof that U_i converges to X.
%4. I think Gromov-Hausdorff distance must be for complete metric spaces because if X' is the completion of X then the GH distance from X to X' must be 0.
%5. I don't like Lemma 5.3 because it seems nontrivial but we don't have a proof. I think it can be removed but then we have to proof something in its place. But this thing is easier since X/Lambda is compact.
%6. I'm wondering if it wouldn't be better to work with mm spaces instead of Riemannian manifolds.
%7. Unimodularity is really needed in the proof of Elek's Theorem. We've just been cavalier about measuring MS(delta). We may even need more hypotheses.
%8. If we do work with mm spaces then we have a homology issue. Do we use singular homology? When is singular homology the same as simplicial?

\section{Introduction}

%Let $\H^n$ denote $n$-dimensional hyperbolic space. It is the simply connected complete Riemannian $n$-manifold with constant curvature $-1$. This paper was initially motivated by:
%\begin{problem}\label{prob:1}
%Does there exist a sequence $\{\Gamma_i\}$ of discrete free subgroups of the isometry group of $\H^n$ such that $\lim_{i\to\infty} h(\H^n/\Gamma_i) = 0$?
%\end{problem}
%Here $h(M)$ denotes the Cheeger constant of a smooth manifold $M$, defined by:
%$$h(M):= \inf \frac{ \textrm{area}(\partial M_0) }{\textrm{vol}(M_0)}$$
%where the infimum is over all smooth compact submanifolds $M_0 \subset M$ with $\vol(M_0) \le \vol(M)/2$. We will mostly be concerned with the case of infinite volume manifolds $M$.

%Observe that for $n=2,3$ there exists a discrete free subgroup $\Gamma <\Isom(\H^n)$ such that $h(\H^n/\Gamma) = 0$. For example, for $n=2$ one can take $\Gamma$ to be a finite-index subgroup of the modular group $PSL_2(\Z) < PSL_2(\R) \approx \Isom(\H^2)$ while for $n=3$ one can take $\Gamma$ to be the fundamental group of a fiber surface in a noncompact finite volume complete hyperbolic $3$-manifold which fibers over the circle. So the answer to Problem \ref{prob:1} is `yes' in dimensions 2 and 3.

The Cheeger constant of a smooth Riemannian manifold $M$ is defined by
$$h(M):= \inf \frac{ \textrm{\area}(\partial M_0) }{\textrm{vol}(M_0)}$$
where the infimum is over all smooth compact submanifolds $M_0 \subset M$ with $\vol(M_0) \le \vol(M)/2$. For most of the paper we will be applying the Cheeger constant to infinite volume manifolds in which case the infimum in the formula above is over all smooth compact submanifolds. In this case, we could call $h(M)$ the {\em F\o lner constant} instead of the Cheeger constant. For example, if $M$ has infinite volume then $h(M)=0$ if and only if $M$ is amenable. This paper is motivated by the following general problem:

\begin{problem}
Given a contractible smooth Riemannian manifold $X$ and a family $\cF$ of abstract groups let $I(X|\cF) = \inf_\Gamma h(X/\Gamma)$ where the infimum is over all $\Gamma< \Isom(X)$ such that 
\begin{itemize}
\item $\Gamma$ acts freely and properly discontinuously on $X$;
\item $\Gamma$ is isomorphic to a group in $\cF$.
\end{itemize}
Compute $I(X|\cF)$ for interesting special cases (e.g., when $X$ is real hyperbolic $n$-space $\H^n$ and $\cF$ is the class of free groups). We are especially interesting in knowing whether $I(X|\cF)=0$.
\end{problem}
For example, let $\Free$ denote the class of free groups. For every $\epsilon>0$ there is a free group $\Gamma<\Isom(\H^2)$ such that the compact core of $\H^2/\Gamma$ is a pair of pants with geodesic boundary components each of length $\epsilon$. The compact core has area $2\pi$ but $\H^2/\Gamma$ has infinite area. It follows that $h(\H^2/\Gamma) \le \frac{3\epsilon}{2\pi}.$ Since $\epsilon$ is arbitrary, $I(\H^2|\Free)=0$. Likewise, $\Isom(\H^3)$ admits a nonuniform lattice isomorphic to the fundamental group of a fiber bundle over the circle so that the fundamental group of the fiber surface is a rank 2 free subgroup $\Lambda$ of $\Isom(\H^3)$ with $h(\H^3/\Lambda)=0$. So $I(\H^3|\Free)=0$. The exact value of $I(\H^n|\Free)$ is unknown for $n>3$. It is not even known whether $I(\H^n|\Free)$ is monotone in $n$.

To explain our main result it is convenient to introduce the following definitions. 
\begin{defn}
Given a Riemannian manifold $X$ and $\Gamma<\Isom(X)$, we say that $\Gamma$ is {\em geometric} if the action of $\Gamma$ on $X$ is free and properly discontinuous. This ensures that $X/\Gamma$ is a manifold and the quotient map $X \to X/\Gamma$ is a cover. 
\end{defn}

%Given a group $\Gamma$, let . 

\begin{defn}
Let us say that a residually finite countable group $\Gamma$ has {\em asymptotically vanishing lower $d$-th Betti number} if 
$$\liminf_N \frac{b_d(N)}{[\Gamma:N]} = 0$$
where the liminf is with respect to the net of finite-index normal subgroups $N \vartriangleleft \Gamma$ ordered by reverse inclusion. Equivalently, this holds if for every $\epsilon>0$, for every finite-index normal subgroup $N \vartriangleleft \Gamma$ there exists a subgroup $N' <N$ such that $N'$ is normal and has finite-index in $\Gamma$ and $\frac{b_d(N')}{[\Gamma:N']} < \epsilon$.
%for every decreasing sequence $\{N_i\}_{i=1}^\infty$ of finite-index normal subgroups of $\Gamma$ with $\cap_{i=1}^\infty N_i = \{e\}$ satisfies 
%$$ \limsup_{i\to\infty} \frac{b_d(N_i)}{[\Gamma:N_i]} =0.$$
For example, if $\Gamma$ has a finite classiyfing space then $\Gamma$ has asymptotically vanishing lower $d$-th Betti number if and only if $b^{(2)}_d(\Gamma)=0$ by L\"uck approximation's Theorem \cite{Lu94} (where $b^{(2)}_d(\Gamma)$ denotes the $d$-dimensional  $L^2$-Betti number of $\Gamma$). 
\end{defn}

Our main result is:
\begin{thm}\label{thm:main-manifold}
Let $X$ be a smooth contractible complete Riemannian manifold.  Let $\cG_d$ be the class of all residually finite countable groups $\Gamma$ such that every finitely generated subgroup $\Gamma' < \Gamma$ has asymptotically vanishing lower $d$-th Betti number. Suppose there is a residually finite geometric subgroup $\Lambda < \Isom(X)$ such that $X/\Lambda$ is compact and $b^{(2)}_d(\Lambda)>0$. Then $I(X| \cG_d)>0$.
\end{thm}

This is derived from a more general result (Theorem \ref{thm:main0}) concerning metric measure spaces. In general, it appears to be a difficult problem to determine whether a given group is in $\cG_d$. However, we show in Proposition \ref{prop:3m} below that if $\Gamma$ is the fundamental group of a complete finite-volume hyperbolic 3-manifold then $\Gamma \in \cG_d$ for all $d>1$ (and the same holds for every subgroup of $\Gamma$). Using this we obtain:
\begin{cor}\label{cor:app1}
If $\Gamma<\Isom(\H^n)$ is geometric and isomorphic with a subgroup of the fundamental group of a complete finite volume hyperbolic 3-manifold, and $n\ge 4$ is an even integer then $h(\H^n/\Gamma) \ge I(\H^n|\cG_{n/2}) >0$. In particular, $I(\H^n|\Free)>0$ for every even integer $n\ge 4$.
\end{cor}
Observe that we do not require the group $\Gamma$ to be finitely generated in the result above.

Instead of the Cheeger constant, one might be interested in the bottom of the spectrum of the Laplace operator of a smooth Riemannian manifold $M$, which we denote by $\lambda_0(M)$.  By \cite{Ch69}
\begin{eqnarray}\label{eqn:Cheeger}
h(M)^2/4 \le \lambda_0(M).
\end{eqnarray}
More precisely, Cheeger proved (\ref{eqn:Cheeger}) when $M$ is compact but it is well-known that it generalizes to the noncompact case. 

In order to compare Corollary \ref{cor:app1} with previous results, recall that a {\em classical Schottky group} is a subgroup of $\Isom(\H^n)$ generated by elements $g_1,\ldots, g_m$ such that there exist pairwise disjoint conformally round balls $B_1, B_2, \ldots, B_m$ and  $B'_1,B'_2, \ldots, B'_m$ in the sphere at infinity $\bS^{n-1}=\partial \H^n$ such that $g_i( \bS^{n-1}/B'_i)  = \textrm{int}(B_i)$ for every $i$. Phillips and Sarnak \cite[Theorem 5.4]{PS85} showed that for every $n\ge 4$ there is a constant $f_n>0$ such that if $\Gamma$ is any classical Schottky subgroup of $\Isom(\H^n)$ then  $\lambda_0(\H^n/\Gamma) \ge f_n$ where $\lambda_0$ denotes the bottom of the spectrum of the Laplace operator. This result was extended by Doyle \cite{Do88} to $n=3$. No such bound exists for $n=2$.

Classical Schottky groups are free groups so it makes sense to ask whether these results hold for free groups more generally. Corollary \ref{cor:app1} and (\ref{eqn:Cheeger}) show that indeed  $\lambda_0(\H^n/\Gamma)\ge I(\H^n|\cG_{n/2})^2/4>0$ whenever $n \ge 4$ is even and $\Gamma$ is a free group. %The proof of Corollary \ref{cor:app1} is completely different from the proofs in \cite{PS85,Do88}.

%Moreover, if one allows the balls $B_i$ in the construction above to be not-necessarily-round topological balls then there is no such bound for $n=3$. 

% We shall obtain:% By \cite[Theorems 1.2, 7.1]{Bu82} if the Ricci curvature of $M$ is bounded from below by $-(\textrm{dim}(M)-1)\delta^2$  then
%\begin{eqnarray}\label{eqn:Buser}
%\lambda_0(M) \le  c(\delta h(M) + h(M)^2)
%\end{eqnarray}
%where $c>0$ is a constant which depends only on $\textrm{dim}(M)$. The curvature condition is not the weakest possible hypothesis under which (\ref{eqn:Buser}) holds but it suffices for our purposes.

Instead of the Cheeger constant or $\lambda_0$, one might be interested in the Hausdorff dimension of the limit set. The limit set $L\Gamma$ of a subgroup $\Gamma < \Isom(\H^n)$ is the intersection of the sphere at infinity $\bS^{n-1}=\partial \H^n$ with the closure of $\Gamma x$ for any $x\in \H^n$. Let $\textrm{HD}(L\Gamma)$ denote the Hausdorff dimension of $L\Gamma$. In \cite[Theorem 2.21]{Su87}, Sullivan shows that if $\Gamma<\Isom(\H^n)$ is geometrically finite without cusps and $\textrm{HD}(L\Gamma) \ge (n-1)/2$ then
\begin{eqnarray}\label{eqn:PS}
\lambda_0(\H^n/\Gamma) = (n-1-\textrm{HD}(L\Gamma))\textrm{HD}(L\Gamma).
\end{eqnarray}
%where $\textrm{HD}(L\Gamma)$ is the Hausdorff dimension of the limit set of $\Gamma$. 
(Our definition of $\lambda_0$ differs from the definition in \cite{Su87} by a sign). If $\Gamma$ is merely geometrically finite then this result holds with the limit set replaced by the conical limit set by \cite[Corollary 2.6]{BJ97} and \cite[Theorem 2.17]{Su87}. These results has been generalized to other rank 1 symmetric spaces in \cite{CI99}. From Corollary \ref{cor:app1} and (\ref{eqn:Cheeger}) we now obtain:

\begin{cor}\label{cor:app2}
For every integer $n\ge 2$, there exists a number $d_{2n} < 2n-1$ such that if $\Lambda<\Isom(\H^{2n})$ is a geometrically finite discrete group isomorphic to a subgroup of the fundamental group of a finite-volume complete hyperbolic 3-manifold then the Hausdorff dimension of the conical limit set of $\Lambda$ is at most $d_{2n}$.
\end{cor}
This corollary partially solves \cite[Problem 10.27, page 530]{Ka08}. Let us mention in passing that the survey  \cite{Ka08} is a rich source for examples and open problems about Kleinian groups in higher dimensions.

It is a long-standing open problem to determine whether there exists a closed real hyperbolic 4-manifold $M$ that fibers over a surface with fiber a surface. Recently U. Hamenst\"adt has proven that no such manifold exists if both base and fiber are closed \cite{Ha13}. However, the other cases remain open. If there is such a manifold, then the universal cover $\widetilde{M}$ is naturally identifiable with hyperbolic space $\H^4$ and therefore the fundamental group $\pi_1(M)$ can be represented as a lattice in $\Isom(\H^4)$. Moreover, the fundamental group of a fiber surface can be represented as a discrete group $\Lambda < \Isom(\H^4)$. This group is isomorphic to the fundamental group of a surface. So Corollary \ref{cor:app1} implies $h(\H^4/\Lambda)>0$. Because $\Lambda$ is a normal subgroup of a lattice, its limit set is the entire 3-sphere boundary of $\H^4$. However, it is not geometrically finite. It might seem reasonable, by analogy with the 3-dimensional case,  to suspect that by deforming $\Lambda$ slightly (or by passing to a subgroup), it should be possible to find, for every $\epsilon>0$, a geometrically finite discrete group $\Lambda' < SO(4,1)$ such that the Hausdorff dimension of the conical limit set of $\Lambda'$ is at least $3-\epsilon$ and $\Lambda'$ is isomorphic to the fundamental group of a compact surface. Corollary \ref{cor:app2} implies this intuition is incorrect. 

\begin{question}
Let $\rm{Surface}$ denote the class of fundamental groups of closed surfaces of genus $g \ge 2$. Is $I(\H^4| \rm{Surface})$ realized? In other words, does there exist a geometric surface group $\Gamma<\Isom(\H^4)$ such that $h(\H^4/\Gamma) = I(\H^4|\rm{Surface})$? Suppose that there exists a closed hyperbolic $4$-manifold which fibers over a surface and $\Gamma$ is the fundamental group of the fiber surface. Then is it true that $h(\H^4/\Gamma) = I(\H^4|\rm{Surface})$? This question admits natural variations by replacing the Cheeger constant with $\lambda_0$ or $\textrm{HD}(L\Gamma)$ for example.
\end{question}

\begin{question}
Is $I(\H^n | \Free)>0$ when $n$ is odd? Does the limit $\lim_{n\to\infty} I(\H^n|\Free)$ exist? If so, is it positive?
\end{question}

\begin{remark}
Corollaries \ref{cor:app1} and \ref{cor:app2} can be generalized to complex hyperbolic space (by using \cite[Theorem 5.12]{Lu02} and Proposition \ref{prop:3m} below). In fact, complex-hyperbolic manifolds are always even dimensional and it is known that the $L^2$-Betti number of a lattice acting complex hyperbolic space does not vanish in the middle dimension. Therefore, we obtain $I(\C\H^n|\cG_n)>0$ for all $n\ge 2$. 
\end{remark}

\begin{remark}
There are stronger results for quaternionic hyperbolic space and the octonionic hyperbolic plane because the isometry groups of these spaces have property (T) \cite{Co90, CI99}. In fact it is known that if $\Gamma$ is a geometrically finite subgroup of the isometry group of one of these spaces but $\Gamma$ is not a lattice then there is a nontrivial lower bound on the codimension of the Hausdorff dimension of the limit set of $\Gamma$ which does not depend on $\Gamma$. The analogous statement for real or complex hyperbolic $n$-space is false \cite{Ka08} essentially because there exist lattices which surject onto infinite amenable groups.
\end{remark}

%\begin{remark}

%\end{remark}

%\begin{defn}
%Let $M$ be a smooth manifold and $x \in M$. The injectivity radius of $M$ at $x$ is the least distance from $x$ to the cut locus at $x$. In other words, it is the largest radius for which the exponential map at $x$ is a diffeomorphism.
%\end{defn}

%\begin{defn}
%A group $\Gamma$ is {\em coherent} if every finitely generated subgroup is finitely presented.
%\end{defn}

\begin{remark}
It is an open question whether Theorem \ref{thm:main-manifold} holds without the residual finiteness assumptions (either on $\Lambda$ or $\cG_d$). However, in many interesting cases $\Isom(X)$ is linear and therefore, by Mal\'cev's Theorem, every finitely generated subgroup of $\Isom(X)$ is residually finite. 
%It does hold in dimension 2 and the proof is essentially the same but in the interest of keeping the paper short we have omitted it.
\end{remark}

%\begin{remark}
%The requirement in Theorem \ref{thm:main-manifold} that every finitely generated subgroup $\Gamma'<\Gamma$ has a finite classifying space and $b^{(2)}_d(\Gamma')=0$ can be relaxed as follows. Let us say that $\Gamma'$ has {\em asymptotically vanishing $d$-th Betti number} if for every decreasing sequence $\{N_i\}_{i=1}^\infty$ of finite-index normal subgroups of $\Gamma'$ with $\cap_{i=1}^\infty N_i = \{e\}$ satisfies 
%$$ \limsup_{i\to\infty} \frac{b_d(N_i)}{[\Gamma':N_i]} =0.$$
%Theorem \ref{thm:main-manifold} remains true if we merely require each $\Gamma \in \cG_d$ to have asymptotically vanishing $d$-th Betti number instead of 

%can be relaxed to stating that Kazhdan's inequality holds for $\Gamma'$ in dimension $d$. To be precise, this means that if $\{N_i\}_{i=1}^\infty$ is a decreasing sequence of finite-index normal subgroups of $\Gamma'$ with $\cap_{i=1}^\infty N_i = \{e\}$ then 
%$$b^{(2)}_d(\Gamma') \ge \limsup_{i\to\infty} \frac{b_d(N_i)}{[\Gamma':N_i]}.$$
%For example, this inequality is known to hold for all finitely generated groups $\Gamma'$ when $d=1$ by \cite[LO10]. 
%
%\end{remark}

\begin{remark}
It might be possible to obtain an explicit bound on $I(\H^{2n}|\cG_n)$ from the proof of Theorem \ref{thm:main-manifold} and the results of \cite{El10} which show Betti numbers are testable. 
\end{remark}

%Theorem \ref{thm:main-manifold} is derived from a result about metric measure spaces in place of Riemannian manifolds. To explain this, we first need some definitions. Notation is explained more completely in the text.
%\begin{defn}
%Let $X$ be a metric measure space and $r>0$. We define the {\em radius-$r$ Cheeger constant of $X$} by
%$$h_r(X) = \inf_M \frac{ \vol_X( N_r(\partial M))}{\vol_X(M)}$$
%where the infimum is over all pathwise-connected compact subsets $M \subset X$, $\partial M = M \cap \overline{X \setminus M}$ and $N_r(\partial M)$ is the radius-$r$ neighborhood of $M$.
%\end{defn}

%\begin{defn}
%For any class of groups $\cF$, metric measure space $X$ and $r>0$ let $I_r(\cF|X)=\inf_\Gamma h_r(X/\Gamma)$ where the infimum is over all geometric $\Gamma < \Isom(X)$ such that $\Gamma$ is isomorphic to a group in $\cF$.
%\end{defn}

%The main result from which Theorem \ref{thm:main-manifold} is derived from is:

%\noindent {\bf Theorem \ref{thm:main0}}
%{\it Let $X$ be a contractible special metric measure space (Definition \ref{defn:special}). Suppose:
%\begin{itemize}
%\item there exists a cocompact residually finite lattice $\Lambda < \Isom(X)$ and $b^{(2)}_d(\Lambda)>0$;
%\item there exists an $\epsilon>0$ such that
%\begin{itemize}
%\item every ball of radius $\le 10\epsilon$ in $X/\Lambda$  is strongly convex;
%\item the volumes $\vol_X(B_X(r,p))$ vary continuously in $r>0$ and $p\in X$. %Moreover $\vol_X(B_X(r,p))>0$ for every $r>0$, $p\in X$. 
%\end{itemize}
%\end{itemize}%
%Then there exists $r>0$ such that $I_r(X|\cG_d)>0$.}

\subsection{Outline}

We begin by explaining Benjamini-Schramm convergence of simplicial complexes in \S \ref{sec:BS}. The highlight of this section is G. Elek's result: if $\{K_i\}_{i=1}^\infty$ is a Benjamini-Schramm-convergent sequence of finite connected simplicial complexes then the normalized Betti numbers of $\{K_i\}_{i=1}^\infty$ converge as $i\to\infty$. This result is the key to the whole proof. In \S \ref{sec:mm} we review metric measure spaces, deferring the proofs to the appendix. We generalize G. Elek's result in \S \ref{sec:Elek} to sequences of metric measure spaces following an outline provided by G. Elek in the closed Riemannian manifold case \cite{El12}. \S \ref{sec:L2} reviews $L^2$-Betti numbers and \S \ref{sec:unimodular} provides some tools from proving Benjamini-Schramm convergence. 

The proof of Theorem \ref{thm:main-manifold} is in \S \ref{sec:main}. Here is a brief and rough outline. It suffices to prove the contrapositive: that if $\{\Gamma_i\}_{i=1}^\infty$ is a sequence of residually finite geometric subgroups of $\Isom(X)$ and $\lim_{i\to\infty} h(X/\Gamma_i) \to 0$ then for all but finitely many $i$ there exist subgroups $\Gamma'_i<\Gamma_i$ such that $b^{(2)}_d(\Gamma'_i)>0$. We are assuming the existence of a residually finite uniform lattice $\Lambda < \Isom(X)$ with $b_d^{(2)}(\Lambda)>0$. We use a lemma due to Buser to find compact smooth submanifolds $M_i \subset X/\Gamma_i$ such that for every $r>0$ the ratio $\frac{\vol(N_r(\partial M_i))}{\vol(M_i)}$ tends to zero as $i\to\infty$ where $N_r(\partial M_i)$ denotes the radius $r$ neighborhood of the boundary of $M_i$.  After passing to a subgroup of $\Gamma_i$ if necessary, we can also require that $M_i$ has ``no short homotopically nontrivial loops''. From these results we conclude that $\{M_i\}_{i=1}^\infty$ ``Benjamini-Schramm converges to $X$''. So our generalization of Elek's result implies $\lim_{i\to\infty} \frac{b_d(M_i)}{\vol(M_i)} = \frac{b_d^{(2)}(\Lambda)}{\vol(X/\Lambda)}$ where $b_d(M_i)$ denotes the ordinary $d$-th Betti number of $M_i$.

The Mayer-Vietoris sequence is employed to show (roughly speaking) that the normalized Betti numbers $\frac{b_d(M_i)}{\vol(M_i)}$ are asymptotically bounded by the normalized Betti numbers of $\Gamma_i$. L\"uck approximation and residual finiteness allow us to replace ordinary Betti numbers with $L^2$-Betti numbers and to compare these limits with the $L^2$-Betti numbers of the lattice $\Lambda$, proving Theorem \ref{thm:main-manifold}. In the last section \S \ref{sec:app}, we use treeability, almost treeability and well-known results about $L^2$ Betti numbers of hyperbolic lattices to obtain Corollaries \ref{cor:app2} and \ref{cor:app1} from Theorem \ref{thm:main-manifold}.

{\bf Acknowledgements}. I am grateful to Miklos Ab\'ert for an excellent lecture on $L^2$-Betti numbers, to G\'abor Elek for sharing a rough draft of a proof of the closed manifold case of Theorem \ref{thm:Elek} below and to Naser Zadeh for pointing out errors in previous versions. Part of this work was inspired by discussions and talks at the AIM workshop ``L2 invariants and their relatives for finitely generated groups'' in Palo Alto, CA September 2011. The author is supported in part by NSF grant DMS-0968762 and NSF CAREER Award DMS-0954606.

\section{Benjamini-Schramm convergence of simplicial complexes}\label{sec:BS}

A {\em rooted simplicial complex} is a pair $(K,v)$ where $K$ is a simplicial complex and $v$ is a vertex of $K$. We say $(K_1,v_1)$ and $(K_2,v_2)$ are {\em root-isomorphic} if there is an isomorphism from $K_1$ to $K_2$ which takes $v_1$ to $v_2$. We let $[K,v]$ denote the root-isomorphism class of $(K,v)$.

Let $\RSC$ denote the set of all root-isomorphism classes of connected rooted locally finite simplicial complexes. 

%We will usually abuse notation by referring to elements of $\RSC$ as rooted simplicial complexes instead of root-isomorphism classes.

We define a topology on $\RSC$ as follows. Given a finite rooted simplicial complex $(L,w)$ and an integer $r>0$, let $U_r(L,w) \subset \RSC$ be the set of all $[K,v] \in \RSC$ such that the closed ball of radius $r$ centered at $v$ in $K$ is root-isomorphic to $(L,w)$. Here we are employing a standard convention: the closed ball of radius $r$ is the subcomplex consisting of all simplices $\sigma$ in $K$ with the property that every vertex $v'$ of $\sigma$ is of distance at most $r$ from $v$ with respect to the path metric on the 1-skeleton of $K$. 

We give $\RSC$ a topology by declaring each $U_r(L,w)$ to be a closed set. For $\Delta>0$ let $\RSC(\Delta) \subset \RSC$ denote the set of all root-isomorphism classes of connected rooted simplicial complexes $[K,v]$ so that every vertex of $K$ has degree at most $\Delta$. With the subspace topology, $\RSC(\Delta)$ is compact and metrizable. Moreover, each $U_r(L,w) \cap \RSC(\Delta)$ is a clopen subset of $\RSC(\Delta)$.  

\begin{defn}
In general, if $X$ is a topological space, we let $\cM(X)$ denote the space of all Borel measures on $X$ with the weak* topology. Therefore a sequence $\{\lambda_i\}_{i=1}^\infty \subset \cM(X)$ converges to an element $\lambda_\infty \in \cM(X)$ if and only if: for every compactly supported continuous function $f \in C(X)$, $\int f~d\lambda_i$ converges to $\int f~d\lambda_\infty$ as $i\to\infty$. Also let $\cM_1(X)$ denote the subspace of Borel probability measures on $X$.
\end{defn}

Given a finite connected simplical complex $K$, let $\mu_{K} \in \cM_1(\RSC)$ denote
$$\mu_K= \frac{1}{|V(K)|} \sum_{v\in V(K)} \delta_{[K,v]}$$
where $V(K)$ denotes the set of vertices of $K$ and $\delta_{[K,v]}$ denotes the Dirac probability measure concentrated on $[K,v] \in \RSC$. 

A sequence of finite  connected simplicial complexes $\{K_i\}_{i=1}^\infty$ is {\em BS-convergent} ({\em Benjamini-Schramm convergent}) if the sequence $\{\mu_{K_i}\}_{i=1}^\infty$ converges in the weak* topology on $\cM_1(\RSC)$. In the special case in which there is a uniform degree bound $\Delta$ on the $K_i$'s, this means that for every finite $(L,w) \in \RSC$ and every $r>0$, $\lim_{i\to\infty} \mu_{K_i}(U_r(L,w))$ exists. The graph-theoretic version of this notion was introduced in \cite{BS01}. The next lemma is crucial to our entire approach.

\begin{lem} \label{lem:elek-simplicial}
Let $\Delta>0$ and $\{K_i\}^\infty_{i=1}$ be a sequence of finite connected simplicial complexes such that every vertex of every $K_i$ has degree at most $\Delta$. If $\{K_i\}_{i=1}^\infty$ is BS-convergent then
$\lim_{i\to\infty} \frac{b_d(K_i)}{|V(K_i)|}$ exists for any $d\geq 0$ where $b_d(K_i)$ denotes the ordinary $d$-th Betti number of $K_i$.
\end{lem}

\begin{proof}This is \cite[Lemma 6.1]{El10}.\end{proof}

%We also let $\RSC = \cup_{d \ge 1} \RSC(\Delta)$ with the induced topology.

The next lemma is a generalization of the above to convex sums of finite connected simplicial complexes.

\begin{lem}\label{lem:Elek2}
Let $\{\eta_i\}_{i=1}^\infty \in \cM_1(\RSC(\Delta))$ be a convergent sequence in the weak* topology. In addition, assume that for each $i$ there exist finite connected simplicial complexes $K_{i,1}, \ldots, K_{i,m_i}$ and positive real numbers $t_{i,1},\ldots, t_{i,m_i}$ such that
$$\eta_i = \sum_{j=1}^{m_i} t_{i,j} \mu_{K_{i,j}}.$$
Suppose as well that there exist natural numbers $N_i$ such that $|V(K_{i,j})| \ge N_i$ for all $i,j$ and $\lim_{i\to\infty} N_i = +\infty$. Then for any $d\ge 1$,
$$\lim_{i\to\infty} \frac{ \sum_{j=1}^{m_i} t_{i,j} b_d(K_{i,j}) }{ \sum_{j=1}^{m_i} t_{i,j} |V(K_{i,j})| }$$
exists.
\end{lem}

\begin{proof}
By approximating the coefficients $t_{i,j}$ by rational numbers, we see that it suffices to prove the special case in which each $t_{i,j}$ is a rational number, which we now assume. Let $D_i>0$ be a natural number such that $D_i t_{i,j} \in \N$ for all $i,j$.

%For each $i$, let $L_i$ be the disjoint union of $D_i t_{i,j}$ disjoint copies of $K_{i,j}$ for $1\le j \le m_i$. 

Let $K_{i,j}^{(1)}, \ldots, K_{i,j}^{(D_i t_{i,j} )}$ be disjoint complexes each of which is isomorphic to $K_{i,j}$. Let $v_{i,j}^k$ be a vertex of $K_{i,j}^{(k)}$. Let $L_i$ be the disjoint union of $K_{i,j}^{(k)}$ over all $1 \le k \le D_i t_{i,j}$ and $1 \le j \le m_i$. Let $L'_i$ be the smallest complex containing $L_i$ such that there is an edge in $L'_i$ from $v_{i,j}^k$ to $v_{i,j}^{k+1}$ for all $1 \le k < D_i t_{i,j}, 1\le j \le m_i$ and an edge from $v_{i,j}^{D_i t_{i,j}}$ to $v_{i,j+1}^1$ for all $1 \le j < m_i$. Then $L'_i$ is a connected complex with vertex degree bound $\Delta+2$. Moreover,
$$b_d(L'_i) = \sum_{j=1}^{m_i} D_i t_{i,j} b_d(K_{i,j}), \quad |V(L'_i)| = \sum_{j=1}^{m_i} D_i t_{i,j} |V(K_{i,j})|$$
which implies 
$$b_d(L'_i)/|V(L'_i)| = \frac{ \sum_{j=1}^{m_i} t_{i,j} b_d(K_{i,j}) }{ \sum_{j=1}^{m_i} t_{i,j} |V(K_{i,j})| }.$$
So it suffices to show that $\lim_{i\to\infty} b_d(L'_i)/|V(L'_i)|$ exists. By Lemma \ref{lem:elek-simplicial}, it suffices to show that $\{L'_i\}_{i=1}^\infty$ is BS-convergent.

Let $(A,a)$ be a finite rooted simplicial complex, $r \in \N$ and, as above, let $U_r(A,a) \subset \RSC$ be the set of all $[K,v] \in \RSC$ such that the closed ball of radius $r$ centered at $v$ in $K$ is root-isomorphic to $(A,a)$. It suffices to show that $\mu_{L'_i}(U_r(A,a))$ converges as $i\to\infty$. However, we observe that $|\mu_{L'_i}(U_r(A,a)) - \eta_i(U_r(A,a))| \le 2|X_i|/|V(L'_i)|$ where $X_i \subset V(L'_i)$ is the set of vertices at distance $\le r$ from the set $\{v_{i,j}^k\}_{j,k} \subset V(L'_i)$. Since $\{\eta_i\}_{i=1}^\infty$ is convergent by hypothesis, it suffices to show that $\lim_{i\to\infty} |X_i|/|V(L'_i)| = 0$.

Because the vertex degrees of $L'_i$ are bounded by $\Delta+2$, it follows that 
$$|X_i| \le (\Delta+2)^r |\{v_{i,j}^k\}_{j,k}| \le (\Delta+2)^r \sum_{j=1}^{m_i} D_i t_{i,j}.$$
On the other hand, 
$$|V(L'_i)| = \sum_{j=1}^{m_i} D_i t_{i,j} |V(K_{i,j})| \ge N_i  \sum_{j=1}^{m_i} D_i t_{i,j}.$$
So 
$$\frac{ |X_i| }{|V(L'_i)|} \le (\Delta+2)^r/N_i$$
which implies $\lim_{i\to\infty} |X_i|/|V(L'_i)| = 0$ as required.
\end{proof}

\section{Metric measure spaces}\label{sec:mm}
In \S \ref{sec:Elek} we generalize Elek's Theorem (Lemma \ref{lem:elek-simplicial} above) by replacing the space of rooted simplicial complexes with the space of pointed metric measure spaces.  In this section, we present the basic definitions and results we will need. The standard reference for this subject is \cite{Gr99}. Our definition of mm$^n$-spaces, given below, and the topology on $\M^n$ appears to be non-standard (at least we did not find it in the literature). We should also mention that Benjamini-Schramm convergence of random length spaces first appeared in \cite{AB12}. Our notion is similar, although not exactly the same.

\begin{defn}
An {\em mm-space} (or metric measure space) is a triple $(M,\dist_M,\vol_M)$ where $(M,\dist_M)$ is a complete separable proper metric space and $\vol_M$ is a Radon measure on $M$. We will usually denote such a space by $M$ leaving $\dist_M$ and $\vol_M$ implicit.  A {\em pointed metric measure space} is a quadruple $(M,p,\dist_M,\vol_M)$ where $(M,\dist_M,\vol_M)$ is an mm-space and $p\in M$. More generally, a {\em pointed mm$^n$-space} is an $(n+3)$-tuple $(M,p,\dist_M,\vol^{(1)}_M, \ldots, \vol^{(n)}_M)$ where $(M,\dist_M)$ is a complete separable proper metric space, $p\in M$ and $\vol^{(i)}_M$ is a Radon measure on $M$ for every $i$. We will often denote a pointed mm$^n$-space by $(M,p)$ leaving the rest of the data implicit. Two pointed mm$^n$-spaces $(M,p), (M',p')$ are {\em isomorphic} if there is an isometry from $M$ to $M'$ that takes $p$ to $p'$ and $\vol_M^{(i)}$ to $\vol_{M'}^{(i)}$ for $i=1\ldots n$. We let $[M,p]$ denote the isomorphism class of $(M,p)$. Let $\M^n$ denote the set of all isomorphism classes of pointed mm$^n$-spaces. Let $\M=\M^1$. %In particular, a {\em pointed mm-space}  is a pointed mm$^1$-space.
%Let $\M^n_{<\infty}$ denote the subset of pointed mm$^n$-spaces $(M,p)$ such that $\vol_M^{(i)}(M)$ is finite for every $i$. 
\end{defn}

\begin{defn}[A topology on $\M^n$]
We define a topology on $\M^n$ by declaring that a sequence $\{[M_i,p_i]\}_{i=1}^\infty$ converges to $[M_\infty,p_\infty]$ in $\M^n$ if and only if there exist a complete proper separable metric space $Z$ and isometric embeddings $\varphi_i:M_i \to Z$ such that
\begin{itemize}
%\item $\lim_{i\to\infty} \varphi_i(p_i)=\varphi_\inftyp_\infty$;
\item $$\lim_{i \to \infty} (\varphi_i( M_i), \varphi_i(p_i)) = (\varphi_\infty(M_\infty), \varphi_\infty(p_\infty)) $$
in the pointed Hausdorff topology (see Definition \ref{defn:pointed-Hausdorff} in the appendix for the definition of this topology);
\item $\lim_{i\to\infty} (\varphi_i)_*\vol^{(k)}_{M_i} = (\varphi_\infty)_*\vol^{(k)}_{M_\infty}$ (in the weak* topology on $\cM(Z)$) for all $k$.
\end{itemize}
\end{defn}

\begin{thm}\label{thm:Mn}
With the topology above, $\M^n$ is separable and metrizable.
\end{thm}
The proof of this theorem is in the appendix. 
%Next we define the space of all isomorphism classes of mm-spaces with extra data (which consists of a locally finite subset and a function on that set). 

\begin{defn}\label{defn:mu_M}
Every non-null finite volume mm-space $M$ is associated with a measure $\mu_M  \in \cM_1(\M)$ obtained by pushing forward the probability measure $\frac{\vol_M}{\vol_M(M)}$ under the map from $M$ to $\M$ given by $p \mapsto [M,p]$. A sequence $\{M_i\}_{i=1}^\infty$ of non-null finite volume mm-space {\em Benjamini-Schramm converges} if $\{\mu_{M_i}\}_{i=1}^\infty$ converges in the weak* topology on $\cM_1(\M)$.
\end{defn}

%Limit points of sequences $\{\mu_{M_i}\}_{i=1}^\infty$ have a special property called {\em unimodularity}:

\section{A variant of Elek's Theorem}\label{sec:Elek}

The purpose of this section is to prove a version of Lemma \ref{lem:elek-simplicial} for metric measure spaces. We first need some definitions to state the result properly.

\begin{defn}[Special metric measure spaces]\label{defn:special}
Let $M$ be an mm-space. We say $M$ is {\em special} if 
\begin{itemize}
\item $\vol_M$ is non-atomic (i.e. $\vol_M(\{x\})=0$ for every $x\in M$);
\item $\vol_M$ is fully-supported (i.e., $\vol_M(O)>0$ for every nonempty open set $O \subset M$);
\item spheres have measure zero (i.e., for all $p \in M, \epsilon>0$, $\vol_M( \{q \in M:~ \dist_M(p,q) = \epsilon\})=0$;
\item $M$ is pathwise connected.
\end{itemize}
Let $\M_{sp} \subset \M$ denote the subspace of isomorphism classes of pointed special mm-spaces. 
\end{defn}

\begin{defn}
Let $M$ be a metric space. We say that $m$ is a {\em midpoint} of $x,y$ (for $m,x,y \in M$) if $\dist_M(x,m)=\dist_M(m,y)=(1/2)\dist_M(x,y)$. We say a subset $X \subset M$ is {\em strongly convex} if every pair $x,y \in X$ has a unique midpoint $m\in X$. 
%n \cite{Ro70}, it is shown that if $X$ is strongly convex and compact then $X$ is contractible. This will be important in the proof of Theorem \ref{thm:Elek} below.
%there is a unique shortest geodesic from $p$ to $q$ and this geodesic lies in $X$.
\end{defn}

\begin{defn}
Let $M$ be a metric space and $\epsilon>0$. A set $S \subset M$ is {\em $\epsilon$-separated} if $\dist_M(s,s')>\epsilon$ for every $s,s' \in S$ with $s \ne s'$. If $Q \subset M$ then $S$ {\em $\epsilon$-covers} $Q$ if for every $q\in Q$ there is an $s\in S$ such that $\dist_M(q,s)<\epsilon$.
\end{defn}

\begin{defn}
Given a metric space $M$, $p \in M$ and $R>0$, let $B_M(p,R)$ denote the closed ball of radius $R$ centered at $p$. Let $B_M^o(p,R)$ denote the open ball of radius $R$ centered at $p$.
\end{defn}
The main result of this section is:

%The purpose of section \S \ref{sec:Elek} is to prove: 
\begin{thm}\label{thm:Elek}
Let $\{M_i\}_{i=1}^\infty$ be a sequence of finite-volume special mm-spaces. Suppose $\lim_{i\to \infty} \mu_{M_i} = \mu_\infty \in \cM_1(\M_{sp})$ exists. We assume there are constants $\epsilon, v_0,v_1$ such that for every $p\in M_i$ (and every $i=1,2,\ldots$)
\begin{itemize}
\item $v_1>\vol_{M_i}(B_{M_i}^o(p,20\epsilon)) \ge \vol_{M_i}(B^o_{M_i}(p,\epsilon/2)) > v_0>0$,
\item $B^o_{M_i}(p, r )$ is strongly convex for every $r \le 10\epsilon$;
%\item the maximum cardinality of an $\epsilon$-separated subset of $B_{M_i}(p,20\epsilon)$ is at most $\Delta$.
\end{itemize}
Then $\lim_{i\to\infty} \frac{ b_d(M_i)}{ \vol(M_i)}$ exists for every $d \ge 1$ where $b_d(M_i)$ denotes the $d$-th ordinary Betti number of $M_i$.
\end{thm}
The main ideas for the proof of Theorem \ref{thm:Elek} are due to G. Elek \cite{El12}. %We will reduce the problem to Lemma \ref{lem:elek-simplicial}. 

%\vspace{0.1in}
%\noindent {\bf Brief outline}. 

\subsubsection{A brief outline}
First we show how to construct for every rooted special mm-space $(M,p)$ a random discrete subset $S \subset M$ which is $\epsilon$-separated and $3\epsilon$-covering. The main difficulty is showing that this construction can be made to depend continuously on $[M,p]$. Secondly we let $\rho^S: S \to [5\epsilon,6\epsilon]$ be a random map and we consider the nerve complex $K$ of the open covering $B^o_M(s,\rho^S(s))$. To be precise, the vertex set of $K$ is $S$ and a subset $S' \subset S$ spans a simplex in $K$ if $\cap_{s\in S'} B^o_M(s,\rho^S(s)) \ne \emptyset$. Considering the case $M=M_i$ with $M_i$ as in Theorem \ref{thm:Elek}, we see that its random complex $K_i$ has degree bound $\Delta$. Moreover we show that $\{K_i\}_{i=1}^\infty$ is Benjamini-Schramm convergent and $K_i$ is homotopic to $M_i$ (this uses a variant of Borsuk's Nerve Theorem). So we can use Lemma \ref{lem:elek-simplicial} to finish the argument.

%The next three definitions are used only in the proof of Theorem \ref{thm:Elek}. It suggested that the reader skip them on first reading.

\subsubsection{Pointed mm-spaces with a weighted discrete set}
We will use the following definitions as technical tools for proving Theorem \ref{thm:Elek}.

\begin{defn}
A {\em pointed mm-space with a weighted discrete set} is a quadruple $(M,p,S,f)$ where $(M,p)$ is a pointed mm-space, $S \subset M$ is a locally finite set and $f:S \to [0,1]$ is a function. By locally finite we mean that $B_M(p,R) \cap S$ is finite for every $R>0$. Two such spaces $(M,p,S,f), (M',p',S',f')$ are {\em isomorphic} if there is an isomorphism from $(M,p)$ to $(M',p')$ which takes $S$ to $S'$ and $f$ to $f'$. Let $\MSF$ denote the set of all isomorphism classes of pointed mm-spaces with a weighted discrete set. We let $[M,p,S,f] \in \MSF$ denote the isomorphism class of $(M,p,S,f)$.
%Let $\MSF_{<\infty}$ be the set of all $(M,p,S,f) \in \MSF$ such that $\vol_M(M)<\infty$ and $|S|<\infty$.
\end{defn}

\begin{defn}[A topology on $\MSF$]\label{defn:MSF}
Given $[M,p,S,f] \in \MSF$, define $\vol_M^{(2)}$ on $M$ to be the counting measure on $S$ and $\vol_M^{(3)}$ on $M$ to be the atomic measure corresponding to $f$. So
$$\vol_M^{(2)}(E) = |E \cap S|, \quad \vol_M^{(3)}(E) = \sum_{s\in E\cap S} f(s)$$
for any $E\subset M$. This defines an embedding of $\MSF$ into $\M^3$. We give $\MSF$ the induced topology.
\end{defn}

\begin{defn}[Pointed mm-spaces with a discrete set]
A {\em pointed mm-space with a discrete set} is a triple $(M,p,S)$ where $(M,p)$ is a pointed mm-space and $S \subset M$ is locally finite. Two such spaces $(M,p,S), (M',p',S')$ are {\em isomorphic} if there is an isomorphism from $(M,p)$ to $(M',p')$ (as elements of $\M$) which maps $S$ bijectively to $S'$. Let $\MS$ denote the set of all pointed mm-spaces with a discrete set up to isomorphism. We let $[M,p,S]\in \MS$ be the isomorphism class of $(M,p,S)$. There is an obvious projection map $\MSF \to \MS$. We endow $\MS$ with the quotient topology. Alternatively, $\MS$ can be embedded into $\M^2$ by $[M,p,S] \mapsto [M,p,\dist_M, \vol_M, \vol^{(2)}_M]$ where $\vol^{(2)}_M$ is the measure $\vol^{(2)}_M(E) = |E \cap S|$.
\end{defn}

%To begin, we need a few preliminary definitions which will enable us to find nice random discrete subsets of each manifold $M_i$, out of which we will construct approximating simplicial complexes. %We will need to consider pointed mm-spaces with random extra data as described next.

%In this section, we prove Theorem \ref{sec:Elek}. 

\subsubsection{Random discrete subsets of mm-spaces}
The first step in the proof of Theorem \ref{thm:Elek} is to associate  to an mm-space a random discrete subset in a natural way. First we need a few more definitions.

\begin{notation}
Given a random variable $X$, let $\Law(X)$ denote the law of $X$. So $\Law(X)$ is a probability measure on the space of all values of $X$. 
\end{notation}

\begin{defn}
If $(Y,\lambda)$ is a purely non-atomic finite measure space and $k \ge 1$ is an integer then $(Y^k,\lambda^k)$ denotes the direct product of $k$-copies of $(Y,\lambda)$ and $({Y \choose k},{\lambda \choose k})$ denotes the projection of $(Y^k,\lambda^k)$ onto the space of all {\em unordered} subsets of $Y$ of cardinality $k$. Because $\lambda$ is purely non-atomic, this is well-defined: the large diagonal in $Y^k$ has measure zero with respect to $\lambda^k$. A {\em uniformly random subset $S \subset Y$ of cardinality $k$} is a random subset with law equal to ${\lambda \choose k}/|{\lambda \choose k}|$ where $|{\lambda \choose k}|$ denotes the total mass of ${\lambda \choose k}$. 
%$(\frac{\lambda}{\lambda(M)}\right)^k /\textrm{sym}(k)$ where the symmetric group $\textrm{sym}(k)$ acts on $M^k$ by permuting the coordinates. 
\end{defn}

\begin{lem}\label{lem:S}
Let $\epsilon>0$. There exists a continuous map $\cF: \M_{sp} \to \cM_1(\MS)$ such that for any $[M,p] \in \M_{sp}$, if $[M',p',S'] \in \MS$ is random with $\Law([M',p',S'])=\cF([M,p])$  then $[M',p']=[M,p]$ and $S'$ is $\epsilon$-separated and $3\epsilon$-covers $M$ almost surely. Moreover, $\cF$ does not depend on the point $p$ in the following sense. If $[M,p], [M,q] \in \M_{sp}$ and $[M,p,S], [M,q,T] \in \MS$ are random with $\Law([M,p,S])=\cF([M,p]), \Law([M,q,T])=\cF([M,q])$ then $\Law(S)=\Law(T)$.
\end{lem}

\begin{proof}%[Proof of Lemma \ref{lem:S}]
Fix $(M,p)$ be a pointed special mm-space. For $j \in \N$, let $S^M_j$ be a Poisson point process on $M$ of intensity $1$. To be precise $S^M_j$ is a random subset of $M$ characterized by the properties:
\begin{enumerate}
\item if $Q \subset M$ has finite volume then $S^M_j \cap Q$ is uniformly random with cardinality $\eta_{j,Q}$ where $\eta_{j,Q}$ is a discrete Poisson random variable with parameter $\lambda=\vol_M(Q)$. So $\textrm{Prob}(\eta_{j,Q} = n) = \frac{\vol_M(Q)^n \exp(-\vol_M(Q)) }{n!}$ for $n=0,1,2,\ldots$.
\item If $\{Q_i\}_{i=1}^\infty$ are pairwise disjoint Borel subsets of $M$ of finite volume then the random variables $\{S^M_j \cap Q_i\}_{i=1}^\infty$ are jointly independent. \end{enumerate}
Also let $f^M_{j}: S^M_{j} \to [0,1]$ be a random function with law $\Leb^{S^M_j}$ where $\Leb$ denotes Lebesgue measure on the interval $[0,1]$. We require that $\{S^M_{j}\}_{j=1}^\infty$ and $\{f^M_{j}\}_{i,j=1}^\infty$ are jointly independent.

\noindent {\bf Claim 1}: The map $[M,p] \in \M_{sp} \mapsto \Law([M,p,S^M_j, f^M_j]) \in \cM_1(\MSF)$ is continuous for each $j$. 

\begin{proof}[Proof of Claim 1]
Let $\{ [M_i,p_i]\}_{i=1}^\infty \subset \M_{sp}$ be a sequence with $\lim_{i\to\infty} [M_i,p_i] = [M_\infty,p_\infty] \in \M_{sp}$. So there are a complete separable proper metric space $Z$ and isometric embeddings $\varphi_i:M_i \to Z$ (for $1\le i \le \infty$) such that 
$$\lim_{i\to\infty} \varphi_i(M_i,p_i) = \varphi_\infty(M_\infty,p_\infty), \quad \lim_{i\to\infty} (\varphi_i)_*\vol_{M_i} = (\varphi_\infty)_*\vol_{M_\infty}.$$
The first limit above is in the pointed Hausdorff topology and the second is in the weak* topology. These limits imply that the Poisson point process with intensity one with respect to the measure $(\varphi_i)_*\vol_{M_i}$ converges in law to the Poisson point process with intensity one with respect to the measure $(\varphi_\infty)_*\vol_{M_\infty}$. Similarly, if $f'_{ij}$ is defined on $\varphi_i(S^{M_i}_j)$ by $f'_{ij}(\varphi_i(s))=f_j^{M_i}(s)$ then $\Law( \varphi_i(S^{M_i}_j), f'_{ij})$ converges to $\Law(\varphi_\infty(S^{M_\infty}_j), f'_{\infty j})$ which implies $\Law([M_i,p,S^{M_i}_j, f^{M_i}_j])$ converges to $\Law([M_\infty,p,S^{M_\infty}_j, f^{M_\infty}_j])$ as $i\to\infty$.
\end{proof}

The idea behind the proof is to construct a random subset $S^M \subset \cup_{j\in \N} S^M_j$ such that the map $[M,p] \mapsto \Law([M,p,S^M])$ satisfies the conclusions of the lemma. We build $S^M$ in stages. In the first stage, we identify a random subset $T^M_1 \subset S^M_1$ such that $U^M_1:=S^M_1 \setminus T^M_1$ is $\epsilon$-separated. In the $n$-th stage we identify a random subset $T^M_{n} \subset S^M_{n}$ such that if $U^M_n := S^M_n \setminus T^M_n$ then $\cup_{j < n} U^M_j$ is $\epsilon$-separated. Finally we let $S^M = \cup_{j=1}^\infty U^M_j$. Randomness is used in the construction of each $T^M_j$ in order to ensure continuity of the map $[M,p] \mapsto \Law([M,p,S^M])$. Next we present the details.

%Claim 1 implies that the map $(M,p) \in \M_{sp} \mapsto \Law(\{(M,p,S^M_j, f^M_j)\}_{j=1}^\infty) \in \cM_1(\MSF^\N)$ is continuous. 
Let $\phi: [0, \infty) \to [0,1]$ be a continuous function satisfying:
\begin{itemize}
\item $\phi(t) = 1$ if $t \le \epsilon$
\item $\phi(t)=0$ if $t \ge 2\epsilon$.
\end{itemize}

%For now, fix $(M,p) \in \M$ and a sequence $\{(M,p,S^M_j,f^M_j)\}_{j=1}^\infty \in \MSF^\N$. 
For each pair $s,t \in \cup_{j=1}^\infty S^M_j$, let $X(s,t) \in [0,1]$ be a random variable with Lebesgue distribution. We require that the $X(s,t)$'s are jointly independent. Let $T^M_1$ consist of every $s\in S^M_1$ such that there is some $t \in S^M_1$ with $f^M_1(s)\le f^M_1(t)$ and $\phi(\dist_M(s,t)) \ge X(s,t)$. Let $U^M_1 = S^M_1 \setminus T^M_1$. Note that $U^M_1$ is $\epsilon$-separated almost surely. %Also $U^M_1$ and $T^M_1$ depend randomly on the variables $\{X(s,t)\}_{s,t}$ and deterministically on $\{S_j^M, f_j^M\}_{j\in\N}$.

\noindent {\bf Claim 2}: The map $[M,p] \in \M \mapsto \Law([M,p,U^M_1]) \in \cM_1(\MS)$ is continuous.

\begin{proof}[Proof of Claim 2]
Let $\{ [M_i,p_i,S^{M_i}_1,f^{M_i}_1]\}_{i=1}^\infty \subset \MSF$ be a (deterministic) sequence with $\lim_{i\to\infty} [M_i,p_i,S^{M_i}_1,f^{M_i}_1] = [M_\infty,p_\infty,S^{M_\infty}_1,f^{M_\infty}_1] \in \MSF$ and such that $f^{M_\infty}_1$ is injective. So there are a complete separable metric space $Z$ and isometric embeddings $\varphi_i:M_i \to Z$ such that 
$$\lim_{i\to\infty} \varphi_i(M_i,p_i) = \varphi_\infty(M_\infty,p_\infty), \quad \lim_{i\to\infty} (\varphi_i)_*\vol^{(k)}_{M_i} = (\varphi_\infty)_*\vol^{(k)}_{M_\infty}$$
for each $k=1,2,3$ where $\vol^{(2)}_{M_i}, \vol^{(3)}_{M_i}$ are as in Definition \ref{defn:MSF}. By Claim 1, it suffices to show that $\Law(\varphi_i(T^{M_i}_1))$ converges to $\Law(\varphi_\infty(T^{M_\infty}_1))$. 

Suppose that $x_i \in S^{M_i}_1$ and 
$$\lim_{i\to\infty} \varphi_i(x_i) = \varphi_\infty(x_\infty)$$ %\quad \lim_{i\to\infty} \varphi_i(y_i) = \varphi_\infty(y_\infty)$$ 
for some $x_\infty \in S^{M_\infty}_1$. Then $f^{M_i}_1(x_i)$ converges to $f^{M_\infty}_1(x_\infty)$. % and $f^{M_i}_1(y_i)$ converges to $f^{M_\infty}_1(y_\infty)$. 

 Let $W(x_i)$ be the set of all $s \in S^{M_i}_1 \cap B_{M_i}(x_i,2\epsilon)$ such that $f^{M_i}_1(x_i) \le f^{M_i}_1(s)$. The probability that $x_i \in T^{M_i}_1$ is the probability that $\phi(\dist_{M_i}(x_i,s)) \ge X(x_i,s)$ for some $s \in W(x_i)$. Note $W(x_i)$ is finite and $\varphi_i(W(x_i))$ converges to $\varphi_\infty(W(x_\infty))$ as $i\to\infty$ in the Hausdorff topology because $f^{M_\infty}_1$ is injective. Also the values of the functions $f^{M_i}_1$ converge in the sense that if $y_i \in W(x_i)$ and $\lim_{i\to\infty} \varphi_i(y_i)=\varphi_\infty(y_\infty)$ then $f^{M_i}_1(y_i)$ converges to $f^{M_\infty}_1(y_\infty)$. Since $\phi$ is continuous, the probability that $x_i \in T^{M_i}_1$ converges to the probability that $x_\infty \in T^{M_\infty}_1$ as $i\to\infty$. Because $\{x_i\}_{i=1}^\infty$ is arbitrary, this implies the claim.
\end{proof}

For $(M,p) \in \M$, we inductively define $T^M_n, U^M_{n}$ (for $n\ge 2$) by: $T^M_{n}$ consists of every $x \in S^M_{n}$ such that there exists $y \in S^M_n \cup \bigcup_{j<n} U^M_j$ with $f^M_n(x)\le f^M_n(y)$ and  $\phi(\dist_M(x,y)) \ge X(x,y)$. Let $U^M_n = S^M_n \setminus T^M_n$. Note $\cup_{j\le n} U^M_j$ is $\epsilon$-separated almost surely.

\noindent {\bf Claim 3}: The map $[M,p] \mapsto \Law([M,p,U^M_n]) \in \cM_1(\MS)$ is continuous for every $n$.

The proof of this is similar to the proof of Claim 2 so we will skip it. Let $S^M= \bigcup_{j=1}^\infty U^M_j$. Note $S^M$ is $\epsilon$-separated almost surely. We claim that $S^M$ $3\epsilon$-covers $M$ if $M$ is special. To see this, let $q \in M$. Let $n>0$ be an integer and consider the event that $\bigcup_{j< n} U^M_j$ has trivial intersection with $B^o_M(q,3\epsilon)$. Conditioned on this event, the probability that $U^M_{n}$ has nontrivial intersection with $B^o_M(q,3\epsilon)$ is bounded below by the probability that $S^M_{n} \cap B^o_M(q,3\epsilon)$ consists of a single point contained in $B_M(q,\epsilon)$. In particular there is a positive lower bound on this probability (depending on $q$) which is independent of $n$. This uses the hypothesis that $\vol_M$ is fully-supported because $M$ is special. By the law of large numbers then, with probability one, $S^M \cap B^o_M(q,3\epsilon) \ne \emptyset$. This proves $S^M$ $3\epsilon$-covers $M$ as claimed. To finish the lemma, define $\cF([M,p]) := \Law([M,p,S^M])$. The continuity of $\cF$ follows from Claim 3.
\end{proof}

%\begin{lem}\label{lem:sc}
%Let $M$ be a metric space and $\{X_i\}_{i \in I}$ strongly convex subsets of $M$. Then $\bigcap_{i \in I} X_i$ is strongly convex. Moreover, any strongly convex compact subset of $M$ is contractible.
%\end{lem}
%\begin{proof}
%The first statement is obvious. The second is proven in \cite{Ro70}.
%\end{proof}

 %For convenience, we recall the statement of Theorem \ref{thm:Elek}:
 
%\noindent {\bf Theorem \ref{thm:Elek}} {\it
%Let $\{M_i\}_{i=1}^\infty$ be a sequence of complete finite-volume Riemannian $n$-manifolds. Suppose $\lim_{i\to \infty} \mu_{M_i} = \mu_\infty \in \cM_1(\RM)$ exists. We assume there are constants $C,\epsilon, v_0,v_1$ such that for every $p\in M_i$ (and every $i=1,2,\ldots$)
%\begin{itemize}
%\item $v_1>\vol(B_{M_i}(p,10\epsilon)) \ge \vol(B_{M_i}(p,\epsilon)) > v_0>0$,
%\item $B_{M_i}(p, 10\epsilon )$ is strongly convex;
%\item the maximum cardinality of an $\epsilon$-separated subset of $B_{M_i}(p,20\epsilon)$ is at most $C$.
%\end{itemize}
%Then $\lim_{i\to\infty} \frac{ b_d(M_i)}{ \vol(M_i)}$ exists for every $d \ge 1$.}

%\begin{remark}
%This theorem and its proof can be generalized from the class of Riemannian manifolds to the class of metric measure spaces by replacing the Gromov-Hausdorff topology with the Gromov-Prohorov topology. 
%\end{remark}

\begin{proof}[Proof of Theorem \ref{thm:Elek}]
%Let $\epsilon>0$ be such that for every $i$, for every $x \in M_i$, the open ball $B_{10\epsilon}(x)$ is strongly convex.

Let $\MS'$ be the set of all $[M,p,S] \in \MS$ such that there is a unique $s\in S$ with $\dist_M(p,s) \le \dist_M(p,s')$ for all $s' \in S$. Given $[M,p,S] \in \MS'$, let $\rho^S:S \to [5\epsilon,6\epsilon]$ be a random function defined by: 
\begin{itemize}
\item for each $t\in S$, $\Law(\rho^S(t))$ is the normalized Lebesgue measure on the interval $[5\epsilon,6\epsilon]$;
\item the family $\{ \rho^S(t):~t\in S\}$ is jointly independent.
\end{itemize}
In other words, the law of $\rho^S$ is the product measure $(\Leb_{[5\epsilon,6\epsilon]})^S$ where  $\Leb_{[5\epsilon,6\epsilon]}$ denotes Lebesgue measure on the interval $[5\epsilon,6\epsilon]$ normalized to have total mass $1$. Let $\Sigma(M,S,\rho^S)$ be the nerve complex of $\{ B^o_{M}(s, \rho^S(s) ):~ s\in S\}$. To be precise, the vertex set of $\Sigma(M,S,\rho^S)$ is $S$ and for every $S' \subset S$ there is a simplex in $\Sigma(M,S,\rho^S)$ spanning $S'$ if and only if $\cap_{s\in S'} B^o_{M}(s,\rho^S(s))  \ne \emptyset$. Let $v \in S$ be the unique element closest to $p$, $\Sigma(M,S,\rho^S)_v$ be the connected component of $\Sigma(M,S,\rho^S)$ containing $v$ and let $\nu_{M,p,S} = \Law(\Sigma(M,S,\rho^S)_v,v)\in \cM(\RSC)$. 

Let $(K,v)$ be a finite rooted simplicial complex, $r>0$ be an integer and $U_r(K,v)$ be the set of all $[K',v'] \in \RSC$ such that the ball of radius $r$ centered at $v'$ is isomorphic to $(K,v)$ as rooted simplicial complexes.  

\noindent {\bf Claim 1}. The map $[M,p,S] \in \MS' \mapsto \nu_{M,p,S}(U_r(K,v))$ is continuous for every $(K,v), r>0$. 

Note: the reason why we choose the radii $\rho^S$ randomly rather than deterministically is to make this claim true.

%Note: this claim would not be true if the radii $\rho^S$ were obtained deterministically instead of randomly. 
\begin{proof}[Proof of Claim 1]
Let $W_r(K,v)$ be the union of all sets of the form $U_r(K',v')$ where $[K',v'] \in \RSC$ is such that there is a simplicial embedding $\phi:K \to K'$ which maps $v$ to $v'$ and is bijective on the $0$-skeleton. Using inclusion-exclusion, it is possible to express $\nu_{M,p,S}(U_r(K,v))$ as a finite linear combination of numbers of the form $\nu_{M,p,S}(W_r(K',v'))$. So it suffices to show that the map $(M,p,S) \mapsto \nu_{M,p,S}(W_r(K,v))$ is continuous.

So let $\{[M_i,p_i,S_i]\}_{i=1}^\infty \subset \MS'$ be a sequence with $\lim_{i\to\infty} [M_i,p_i,S_i] =[M_\infty,p_\infty,S_\infty] \in \MS'$. Without loss of generality, we may assume there is a complete proper separable metric space $Z$ containing $M_i$ for $1\le i \le \infty$ such that 
\begin{itemize}
\item $\dist_{M_i}$ is the restriction of $\dist_Z$ to $M_i$ (for all $i$);
\item $(M_i,p_i)$ converges to $(M_\infty,p_\infty)$ in the pointed Hausdorff topology;
\item $(S_i,p_i)$ converges to $(S_\infty,p_\infty)$ in the pointed Hausdorff topology.
\end{itemize}
Let $R=100\epsilon r$. Since each $S_i$ is locally finite, there is an integer $n>0$ and $s_{i,1},\ldots, s_{i,n} \in S_i$ such that 
\begin{itemize} 
\item $\lim_{i\to\infty} s_{i,j} = s_{\infty,j}$ for each $j$, 
%\item $s_{i,j} \ne s_{i,k}$ if $j \ne k$;
\item $B_Z(p_i, R) \cap S_i \subset \{s_{i,1},\ldots, s_{i,n}\}$ for all $i$.
\end{itemize}
Let $E_i$ be the set of all $t=(t_1,\ldots, t_n) \in [5\epsilon, 6\epsilon]^n$ such that if $\rho:S_i \to [5\epsilon,6\epsilon]$ is any function with $\rho(s_{i,j})=t_j$ for all $j$ then $(\Sigma(M_i,S_i,\rho)_{v_i}, v_i) \in W_r(K,v)$ where $v_i \in S_i$ is the unique closest point to $p_i$. By definition, $\nu_{M_i,p_i,S_i}(W_r(K,v)) = \Leb_{[5\epsilon,6\epsilon]}^n(E_i)$.

Note that $E_i$ is open (because the nerve complexes are defined in terms of open sets). Also, the definition of $W_r(K,v)$ implies that $E_i$ has  the following monotone property: if $t \in E_i$ and $t' \in [5\epsilon, 6\epsilon]^n$ satisfies $t'_j \ge t_j$ for all $j$ then $t' \in E_i$. In order to estimate the volume of $E_i$, let $f_i:[5\epsilon,6\epsilon]^{n-1} \to [5\epsilon,6\epsilon]$ be the function $f_i(t_1,\ldots t_{n-1})=t_n$ where $t_n$ is the largest number in $[5\epsilon,6\epsilon]$ such $(t_1,\ldots, t_n) \notin E_i$ if such a number exists. Otherwise, set $f_i(t_1,\ldots t_{n-1})=5\epsilon$. Then the complement of $E_i$ is the region below the graph of $f_i$. So
$$\nu_{M_i,p_i,S_i}(W_r(K,v)) = \Leb_{[5\epsilon,6\epsilon]}^n(E_i) = 1-\int f_i(t_1,\ldots, t_{n-1}) ~d\Leb_{[5\epsilon,6\epsilon]}^{n-1}(t_1,\ldots,t_{n-1}).$$
Because $\lim_{i\to\infty} s_{i,j} = s_{\infty,j}$ for each $j$ and $(M_i,p)$ converges to $(M_\infty,p_\infty)$, it follows that $\{f_i\}_{i=1}^\infty$ converges pointwise to $f_\infty$. The Bounded Convergence Theorem now implies $\nu_{M_i,p_i,S_i}(W_r(K,v))$ converges to  $\nu_{M_\infty,p_\infty,S_\infty}(W_r(K,v))$ as $i\to\infty$. %Therefore $(M,p,S) \in \MS' \mapsto \nu_{M,p,S} \in \cM_1(\RSC)$ is continuous as claimed.

\end{proof}

%Given an special mm-space $M$ and $s\in M$ define $\kappa(s)\ge 0$ be the smallest number such that  $\vol_{M}(B_{M}(s,\kappa(s)))=v_0/2$ if such a number exists. Let $\kappa(s)=+\infty$ if no such number exists. Let $\MS(\delta)$ be  the set of all $(M',p',S')\in \MS$ such that $\dist_{M'}(p',s) \le \kappa(s)$ for some $s \in S'$. 

Given a special mm-space $M$, $s\in M$ and $r>0$ let $\kappa(s,r)\ge 0$ be the smallest radius such that  $\vol_{M}(B_{M}(s,\kappa(s,r)))=r$ if such a number exists. Let $\kappa(s,r)=+\infty$ if no such number exists. Let $\MS(r)$ be  the set of all $(M',p',S')\in \MS$ such that $\dist_{M'}(p',s) \le \kappa(s,r)$ for some $s \in S'$. Similarly, let $\MS^o(r)$ be  the set of all $(M',p',S')\in \MS$ such that $\dist_{M'}(p',s) < \kappa(s,r)$ for some $s \in S'$.

%if $r=\dist_{M'}(p',S') = \inf_{s\in S'} \dist_{M'}(p',s)$ then $\vol_{M'}(B_{M'}(s,r)) \le \delta$ for every $s\in S'$ which realizes the minimum distance to $p'$ (so $r=\dist_{M'}(p',s)$). 
%For $\delta>0$, let $\MS(\delta)$ be the set of all $(M',p',S')\in \MS$ such that if $r=\dist_{M'}(p',S') = \inf_{s\in S'} \dist_{M'}(p',s)$ then $\vol_{M'}(B_{M'}(s,r)) \le \delta$ for every $s\in S'$ which realizes the minimum distance to $p'$ (so $r=\dist_{M'}(p',s)$). 

Let $M_i$ be as in the statement of Theorem \ref{thm:Elek}, $p_i \in M_i$ be uniformly random, $S_i \subset M_i$ be such that $\Law([M_i,p_i,S_i]) = \cF([M_i,p_i])$ as in Lemma \ref{lem:S} and $\lambda_i=\Law([M_i,p_i,S_i])$ for $1\le i < \infty$. By the hypotheses of Theorem \ref{thm:Elek} and Lemma \ref{lem:S}, $\lambda_i$ converges as $i\to \infty$ to a measure $\lambda_\infty \in \cM_1(\M_{sp})$. Let $[M_\infty,p_\infty,S_\infty]\in \MS$ be random with law $\lambda_\infty$. By hypothesis, $M_\infty$ is a special mm-space almost surely.

\noindent {\bf Claim 2}. 
%$\lim_{i\to\infty} \lambda_i(\MS(v_0/2)) =\lambda_\infty(\MS(v_0/2)) \ge\frac{v_0}{2v_1}$.
\begin{enumerate}
\item $\lambda_\infty(\partial \MS(v_0/2))=0$ where $\partial \MS(v_0/2) = \overline{\MS(v_0/2)} \cap \overline{\MS \setminus \MS(v_0/2)}$;
\item $\lim_{i\to\infty} \lambda_i(\MS(v_0/2)) =\lambda_\infty(\MS(v_0/2)) \ge\frac{v_0}{2v_1}$.
\end{enumerate}
\begin{proof}[Proof of Claim 2]
Note that for every $s \in M_i$, 
$$\kappa(s,v_0/2) < \kappa(s,v_0) < \epsilon/2$$
because $\vol_{M_i}(B^o_{M_i}(s,\epsilon/2)) > v_0>0$ and because $M_i$ is special so spheres in $M_i$ have measure zero. Because $\lambda_i$ converges to $\lambda_\infty$ and $\MS(r)$ is closed in $\MS$, the Portmanteau Theorem implies
\begin{eqnarray}\label{eqn:closed}
\limsup_{i\to\infty} \lambda_i(\MS(r)) \le \lambda_\infty(\MS(r))\quad \forall r >0.
\end{eqnarray}
Because $\MS^o(r)$ is open in $\MS$,
\begin{eqnarray}\label{eqn:open}
\liminf_{i\to\infty} \lambda_i(\MS^o(r)) \ge \lambda_\infty(\MS^o(r))\quad \forall r >0.
\end{eqnarray}
Now observe that 
$$\lambda_i(\MS^o(r)) = \lambda_i(\MS(r)) = \frac{|S_i| r}{\vol(M_i)}$$
if $r \le v_0$ because spheres in $M_i$ have measure zero, $\kappa(s,v_0)<\epsilon/2$ and $S_i$ is $\epsilon$-separated. In particular, if $0<r_1,r_2 < v_0$ then
\begin{eqnarray*}
\frac{\lambda_\infty(\MS^o(r_1))}{\lambda_\infty(\MS(r_2))} &\le& \frac{\liminf_{i\to\infty} \lambda_i(\MS^o(r_1))}{  \limsup_{i\to\infty} \lambda_i(\MS(r_2))} \le \liminf_{i\to\infty} \frac{\lambda_i(\MS^o(r_1))}{  \lambda_i(\MS(r_2))} = \frac{r_1}{r_2}.
\end{eqnarray*}
Similarly,
\begin{eqnarray*}
\frac{\lambda_\infty(\MS(r_1))}{\lambda_\infty(\MS^o(r_2))} &\ge& \frac{\limsup_{i\to\infty} \lambda_i(\MS(r_1))}{  \liminf_{i\to\infty} \lambda_i(\MS^o(r_2))} \ge \limsup_{i\to\infty} \frac{ \lambda_i(\MS(r_1))}{  \lambda_i(\MS^o(r_2))}= \frac{r_1}{r_2}.
\end{eqnarray*}
So for any sufficiently small $\delta>0$,
\begin{eqnarray*}
\frac{r_1-\delta}{r_2+\delta} &\le& \frac{\lambda_\infty(\MS(r_1-\delta))}{\lambda_\infty(\MS^o(r_2+\delta))} \le \frac{\lambda_\infty(\MS^o(r_1))}{\lambda_\infty(\MS(r_2))}\\ &\le& \frac{\lambda_\infty(\MS(r_1))}{\lambda_\infty(\MS^o(r_2))} \le \frac{\lambda_\infty(\MS^o(r_1+\delta))}{\lambda_\infty(\MS(r_2-\delta))}\le \frac{r_1+\delta}{r_2-\delta}.
\end{eqnarray*}
By sending $\delta \searrow 0$ we see that
$$\frac{r_1}{r_2} = \frac{\lambda_\infty(\MS(r_1))}{\lambda_\infty(\MS^o(r_2))} = \frac{\lambda_\infty(\MS^o(r_1))}{\lambda_\infty(\MS(r_2))}.$$
In particular, $\lambda_\infty(\MS(v_0/2)) = \lambda_\infty(\MS^o(v_0/2))$ which implies $\lambda_\infty(\partial \MS(v_0/2))=0$. By (\ref{eqn:closed}, \ref{eqn:open})
$$\lim_{i\to\infty} \lambda_i(\MS(v_0/2)) =\lambda_\infty(\MS(v_0/2)).$$
Because $B_{M_i}(q,3\epsilon)<v_1$ (for any $q \in M_i$) and $S_i$ is $3\epsilon$-covering, it follows that
$$v_1|S_i|\ge  \vol_{M_i}(M_i).$$
Because $\kappa(s,v_0/2)\le \epsilon/2$ and $S_i$ is $\epsilon$-separated it follows that the collection of balls of radii $\kappa(s,v_0/2)$ centered at $s \in S_i$ is pairwise disjoint. Therefore
$$\lambda_i(\MS(v_0/2)) = \frac{|S_i| v_0/2}{\vol_{M_i}(M_i)} \ge \frac{v_0}{2v_1}>0.$$
So $\lambda_\infty(\MS(v_0/2)) \ge \frac{v_0}{2v_1}>0.$

 %implies that $\lim_{i\to\infty} \lambda_i(\MS(v_0/2)) =\lambda_\infty(\MS(v_0/2)) \ge\frac{v_0}{2v_1}$ as required.
\end{proof}

By Claim 2 and the Portmanteau Theorem, $\lambda_i'$ converges to $\lambda_\infty'$ in the weak* topology as $i\to\infty$ where $\lambda'_i$ denotes the normalized restriction of $\lambda_i$ to $\MS(v_0/2)$. More precisely, 
$$\lambda'_i(E) := \frac{ \lambda_i(E \cap \MS(v_0/2))}{\lambda_i(\MS(v_0/2))}$$
for every Borel $E \subset \MS$. 

%Because $\vol_{M_i}(B_{M_i}(q,\epsilon)) > v_0 \ge \delta$ for every $q \in M_i$ and $M_i$ is special, there exists $0<\kappa(\delta,q)<\epsilon$ such that $\vol_{M_i}(B_{M_i}(q,\kappa(\delta,q))=\delta$. Let $\MS(\delta)$ be the set of all $(M',p',S') \in \MS$ such that there exists $s'\in S'$ with $\dist_{M'}(p',s')\le \kappa(\delta,s')$. 

% and $f_i:S_i \to [5\epsilon,6\epsilon]$ be a random function with law equal $\Leb^{S_i}_{[5\epsilon,6\epsilon]}$ where $\Leb_{[5\epsilon,6\epsilon]}$ denotes normalized Lebesgue measure on $[5\epsilon,6\epsilon]$ (for $1\le i \le \infty$). Note that $\Law(M_i,p_i,S_i,f_i)$ converges to $\Law(M_\infty,p_\infty, S_\infty, f_\infty)$ in $\MSF$.

If $T \subset B_{M_i}(s,20\epsilon)$ is any $\epsilon$-separated subset then because 
$$v_1>\vol_{M_i}(B_{M_i}^o(q,20\epsilon)) \ge \vol_{M_i}(B^o_{M_i}(q,\epsilon/2)) > v_0>0$$
for every $q\in M_i$, we must have $v_0|T| \le v_1$. So $|T| \le v_1/v_0$. So setting $\Delta:=v_1/v_0$, we see that the degree of any vertex in $\Sigma(M_i,S_i,\rho^{S_i})$ is at most $\Delta$. So if $\nu'_i:=\int \nu_{M_i,p_i,S_i}~d\lambda'_i(M_i,p_i,S_i)$ then $\nu'_i \in \cM_1(\RSC(\Delta))$. By Claim 1 and the fact that $\MS(v_0/2) \subset \MS'$, $\lim_{i \to \infty} \nu'_i(U_r(K,v))= \nu'_\infty(U_r(K,v))$ for every finite $(K,v) \in \RSC(\Delta)$ and $r>0$. Because each $\nu'_i \in \cM_1(\RSC(\Delta))$ and the sets $U_r(K,v)$ generate the Borel sigma-algebra of $\RSC(\Delta)$, it follows that $\nu'_i$ converges to $\nu'_\infty$ in the weak* topology as $i\to\infty$.

Let $[K_i,w_i] \in \RSC$ be random with law $\nu'_i$. We claim that the law of $w_i$ given $K_i$ is uniform over the vertex set of $K_i$ (for $1\le i <\infty$). Indeed, the set of vertices of $K_i$ is $S_i$ and $w_i\in S_i$ is the nearest point to $p_i$ when $p_i \in M_i$ is chosen uniformly at random subject to the condition that $\dist_{M_i}(p_i,w_i) \le \kappa(w_i,v_0/2)$. The element $w_i$ is uniquely determined by $p_i$ because $S_i$ is $\epsilon$-separated with $\epsilon/2\ge \kappa(w_i,v_0/2)$. So the balls $B_{M_i}(s, \kappa(s,v_0/2))$ are pairwise disjoint for $s \in S_i$ and each has the same volume, namely $v_0/2$. Therefore $w_i$ is uniformly distributed over $S_i$ as required.

Because each $M_i$ is special, each is pathwise connected. This implies $K_i$ is connected. It now follows from Lemma \ref{lem:Elek2} that
\begin{eqnarray}\label{eqn:Elek}
\lim_{i\to\infty} \frac{\EE[b_k(K_i)]}{ \EE[ |V(K_i)| ]}
\end{eqnarray}
exists, where $\EE[\cdot]$ denotes expected value.

Because $B^o_{M_i}(s, r )$ is strongly convex for every $r \le 10\epsilon$, for any subset $S' \subset S_i$, either $\bigcap_{s \in S'} B^o_{M_i}(s, \rho^{S_i}(s))$ is empty or it is strongly convex. In the latter case, it is contractible by  \cite{Ro70}.  This implies that $K_i$ is homotopy equivalent to $M_i$ by \cite[Corollary 4G.3]{Ha02} (this is a slightly stronger version of Borsuk's Nerve Theorem \cite{Bo48}). So $\EE[b_k(K_i)] = b_k(M_i)$. Because of (\ref{eqn:Elek}) it now suffices to prove that 
$$\lim_{i\to\infty} \frac{\EE[ |V(K_i)| ]}{\vol(M_i)}$$
exists.

Note $|V(K_i)| = |S_i| = \vol( M_i(v_0/2)) (v_0/2)^{-1}$ where $M_i(v_0/2)$ is the set of all $q \in M_i$ such that $\dist_{M_i}(q,s)\le \kappa(s,v_0/2)$ for some $s\in S_i$. So
\begin{eqnarray*}
\lim_{i\to\infty}  \frac{\EE[ |V(K_i)| ]}{\vol(M_i)} &=& (v_0/2)^{-1} \lim_{i\to\infty} \frac{ \EE[\vol(M_i(v_0/2))] }{\vol(M_i)}\\
 &=& (v_0/2)^{-1} \lim_{i\to\infty} \lambda_i( \MS(v_0/2)) = (v_0/2)^{-1}  \lambda_\infty( \MS(v_0/2)).
 \end{eqnarray*}

\end{proof}

The next result is not needed in the sequel. However, it seems worth recording for the sake of future research. This result was first obtained by G. Elek \cite{El12}. 

\begin{defn}
We consider any Riemannian manifold $X$ as an mm-space with distance $\dist_X$ equal to the Riemannian distance and measure $\vol_X$ equal to the Riemannian volume form.
\end{defn}

\begin{cor}
Let $\{M_i\}_{i=1}^\infty$ be a sequence of connected closed smooth Riemannian $n$-manifolds. Suppose that $\{M_i\}_{i=1}^\infty$ Benjamini-Schramm converges in the sense of Definition \ref{defn:mu_M}. Suppose also that there are constants $\delta, \kappa$ such that for each $M_i$, all sectional curvatures are bounded from above by $\kappa$ and all Ricci curvatures are bounded from below by $\delta$. Suppose also that the injectivity radius of $M_i$ tends to infinity as $i\to\infty$. Then the normalized limit
$$\lim_{i\to\infty} \frac{b_d(M_i)}{\vol(M_i)}$$
exists for every $d\ge 1$.
\end{cor}

\begin{proof}
It suffices to check that the conditions of Theorem \ref{thm:Elek} are met. The volume bounds on balls follow from \cite[Theorems 3.7 and 3.9]{Ch93}. Strong convexity of small balls follows from \cite[Theorem 7.9]{Ch93}. The other conditions are trivial to verify.  
\end{proof}

\section{$L^2$-Betti numbers}\label{sec:L2}
 In this section, we quickly review facts about $L^2$-invariants used in the proof of Theorem \ref{thm:main-manifold}. We refer the reader to \cite{Lu02, Lu09} for background.

Given a topological space $X$ with a continuous $\Gamma$-action (where $\Gamma$ is a countable discrete group), we may define the $L^2$-Betti numbers $b_k^{(2)}(X;\mathcal{N}(\Gamma))$ (for $k \in \N$) (where $\mathcal{N}(\Gamma)$ denotes the von Neumann algebra of $\Gamma$). For simplicity, we let $b_k^{(2)}(X)$ denote $b_k^{(2)}(\tilde{X};\mathcal{N}(\pi_1(X)))$ where $\tilde{X}$ is the universal cover of $X$ and $\pi_1(X)$ acts on $\tilde{X}$ in the usual way. These numbers are known to be homotopy invariants. Hence we may define the $L^2$-Betti numbers of a countable discrete group $\Gamma$ by $b_k^{(2)}(\Gamma) := b_k^{(2)}(B\Gamma)$ where $B\Gamma$ is any  classifying space for $\Gamma$ (i.e., $B\Gamma$ is a connected CW-complex with $\pi_1(B\Gamma)$ isomorphic to $\Gamma$ and $\pi_n(B\Gamma) = 0$ for all $n\ge 2$). 

\begin{thm}\label{thm:Luck}
Let $M$ be a finite connected CW-complex. Suppose there is a decreasing sequence $\{N_i\}_{i=1}^\infty$ of finite-index normal subgroups $N_i \vartriangleleft \pi_1(M)$ such that $\cap_{i=1}^\infty N_i = \{e\}$. Let $M_i \to M$ be the finite cover associated to $N_i$. Then for any integer $k\ge 0$,
$$\lim_{i\to\infty} \frac{b_k(M_i)}{[\pi_1(M):N_i]} = b^{(2)}_k(M)$$
where $b^{(2)}_k(M)$ is the $k$-th $L^2$-Betti number of $M$ and $b_k(M_i)$ is the ordinary $k$-th Betti number of $M_i$ (with real coefficients).
%Suppose $\Gamma$ is a finitely generated group and $\{N_i\}_{i=1}^\infty$ is a decreasing sequence of finite-index normal subgroups of $\Gamma$ with $\cap_{i=1}^\infty N_i = \{e\}$. Then 
%$$\lim_{i\to\infty} \frac{b_k(N_i)}{[\Gamma:N_i]} ...$$

\end{thm}

\begin{proof}
This is \cite[Theorem 0.1]{Lu94}. 
\end{proof}

%\begin{thm}\label{thm:finite-index}
%For any countable group $\Gamma$, finite-index subgroup $\Gamma'<\Gamma$ and $d \ge 0$,
%$$ b^{(2)}_d(\Gamma') = [\Gamma:\Gamma'] b^{(2)}_d(\Gamma).$$
%\end{thm}

%\begin{proof}
%This follows from \cite[Theorem 1.35 (9)]{Lu02}. 
%\end{proof}

%\begin{cor}\label{cor:Elek}
%Let $\{M_i\}_{i=1}^\infty$ be as in Theorem \ref{thm:Elek}. Suppose that $\pi_1(M_i)$ is residually finite and finitely generated. Then 
%$$\lim_{i\to\infty} \frac{b_k(M_i)}{\textrm{vol}(M_i)} = \lim_{i\to\infty} \frac{b_k^{(2)}(M_i)}{\textrm{vol}(M_i)}$$
%exists.
%\end{cor}

%\begin{proof}%
%By Theorem \ref{thm:Luck}, for each $i$ there exists a finite cover $\pi_i:M'_i \to M_i$ such that 
%$$\left| \frac{b_k(M'_i)}{deg(\pi_i)} - b_k^{(2)}(M_i) \right| < (1/i)\textrm{vol}(M_i).$$
%Observe that $\{M'_i\}_{i=1}^\infty$ is convergent and moreover, it has the same limit as $\{M_i\}_{i=1}^\infty$. So Theorem \ref{thm:Elek} implies the following limits exist:
%\begin{eqnarray*}
%\lim_{i\to\infty} \frac{b_k(M_i)}{\textrm{vol}(M_i)} &=&\lim_{i\to\infty} \frac{b_k(M'_i)}{\textrm{vol}(M'_i)} = \lim_{i\to\infty} \frac{b_k(M'_i)}{deg(\pi_i) \textrm{vol}(M_i)} =\lim_{i\to\infty} \frac{b_k^{(2)}(M_i)}{\textrm{vol}(M_i)}.
%\end{eqnarray*}
%\end{proof}

\section{Unimodular measures}\label{sec:unimodular}

%The purpose of this section is to develop a tool for proving convergence of sequences in $\cM_1(\M)$. 
Measures of the form $\mu_M$ (where $M$ is a non-null finite volume mm-space) have a special property called {\em unimodularity} which is a kind of statistical homogeneity. We will use this property to prove convergence of certain sequences in $\cM_1(\M)$. To begin we need a few definitions.
% We will need to make use of it to generalize Elek's theorem (Lemma \ref{lem:elek-simplicial}). 

\begin{defn}
A {\em doubly-pointed mm-space} is a quintuple $(M,p,q,\dist_M,\vol_M)$ where $(M,\dist_M,\vol_M)$ is an mm-space and $p,q \in M$. We will usually denote such a space by $(M,p,q)$ leaving the rest implicit. We say $(M,p,q)$ and $(M',p',q')$ are {\em doubly-pointed isomorphic} if there is an isometry from $M_1$ to $M_2$ which takes $p$ to $p'$, $q$ to $q'$ and $\vol_M$ to $\vol_{M'}$. Let $\DM$ denote the set of all isomorphism classes of doubly-pointed mm-spaces.  We let $[M,p,q] \in \DM$ denote the isomorphism class of $(M,p,q)$. We can embed $\DM$ into $\M^2$ by $[M,p,q,\dist_M,\vol_M] \mapsto [M,p,\dist_M,\vol_M, \delta_q]$ where $\delta_q$ is the Dirac probability measure concentrated on $\{q\}$. We give $\DM$ the induced topology.
\end{defn}

\begin{defn}
%Define maps $\pi^l, \pi^r
Let  $\lambda \in \cM_1(\M)$. Define measures $\lambda_l,\lambda_r$ on $\DM$ by
$$d\lambda_l ([M,p,q]) = d\vol_M(q) d\lambda([M,p]), \quad d\lambda_r([M,p,q]) = d\vol_M(p) d\lambda([M,q])$$
%$$\int f  ~d\lambda_l  = \iint f(M,p,q)~d\vol_M(q) ~d\lambda(M,p)$$
%$$\int f  ~d\lambda_r  = \iint f(M,p,q)~d\vol_M(p) ~d\lambda(M,q)$$
%for any bounded Borel function $f$ on $\DM$. 
For example, this means that if $f$ is a positive Borel function on $\DM$ then
$$\int f([M,p,q]) ~d\lambda_l ([M,p,q]) = \int f([M,p,q]) ~d\vol_M(q) d\lambda([M,p]).$$
We say that $\lambda$ is {\em unimodular} if $\lambda_l = \lambda_r$. This term originally appeared in percolation theory (see e.g.,  \cite{AL07} and the references therein). 
\end{defn}

\begin{example}
Let $M$ be a non-null finite volume mm-space and $p \in M$ be a uniformly random point. Then $\Law([M,p])=\mu_M \in \cM_1(\M)$ is unimodular. Assuming $M$ is connected, let $\widetilde{M}$ be the universal cover of $M$ and let $\tilde{p} \in \widetilde{M}$ be an inverse image of $p$. The pointed-isometry class of $(\widetilde{M},\tilde{p})$ does not depend on the choice of $\tilde{p}$. Also $\Law(\widetilde{M},\tilde{p})$ is unimodular.
\end{example}

\begin{lem}\label{lem:closed}
The space of unimodular measures in $\cM_1(\M)$ is closed in $\cM_1(\M)$.
\end{lem}

\begin{proof}
%The maps $\lambda \mapsto \lambda_l$ and $\lambda \mapsto \lambda_r$ are continuous as maps from $\cM_1(\M)$ to $\cM_1(\M\times \M)$. The space of unimodular measures is the inverse image of the set 

Let $\pi: \cM_1(\M) \to \cM(\DM) \times \cM(\DM)$ be the map $\pi(\lambda) = (\lambda_l, \lambda_r)$. This is a continuous map. Since the space of unimodular measures is $\pi^{-1}( \{ (\lambda_1,\lambda_2):~ \lambda_1=\lambda_2\})$, it must be closed in $\cM_1(\M)$.
\end{proof}

\begin{remark}
Let $\cF \subset \cM_1(\M)$ be the space of all measures of the form $\mu_M$ where $M$ is a finite-volume mm-space and $\mu_M=\Law([M,p])$ where $p\in M$ is uniformly random. The relative closure $\overline{\cF} \cap \cM_1(\M) \subset \cM_1(\M)$ is the space of {\em sofic measures}.  Are all unimodular measures sofic? This question is a generalization of the well-known problem: are all groups sofic? It is also a generalization of the problem: are all unimodular networks sofic? This was first asked in \cite{AL07}.
%Another open question is whether
%Let us say that a measure $\lambda \in \cM_1(\RM)$ is {\em sofic} if it 
% if there exists a sequence $\{M_i\}_{i=1}^\infty$ of finite-volume complete Riemannian manifolds such that $\lambda = \lim_{i\to\infty} \mu_{M_i}$ (where $\mu_{M_i}$ is as defined in \S ..). It follows from Theorem \ref{thm:Elek} that for any $d\in \N$, the map $\mu_M \mapsto \frac{b_d(M)}{\vol(M)}$ from the space of measures in $\cM_1(\RM)$ of the form $\mu_M$ (for some finite volume complete Riemannian manifold $M$) extends to a continuous function on the space of sofic measures. It follow from Lemma \ref{lem:closed} that the space of sofic measures is closed. This leads to an interesting open question: can this function be extended 
\end{remark}

\begin{defn}
If $X$ is a metric measure space then $\Isom(X)$ denotes the group of all meaure-preserving isometries $\phi:X \to X$. To be precise, we require $\phi_*\vol_X=\vol_X$. A subgroup $\Lambda<\Isom(X)$ is a {\em lattice} if there exists a measurable subset $\Delta \subset X$ of positive finite volume such that $\{\gamma \Delta:~\gamma\in \Lambda\}$ is a partition of $X$. Such a set is called a {\em fundamental domain} for $\Lambda$. 
\end{defn}

\begin{lem}\label{lem:X}
Let $X$ be a pathwise connected mm-space. Suppose there is a lattice $\Lambda <\Isom(X)$. Then there is a unique unimodular measure $\mu \in \cM_1(\M)$ such that $\mu$-almost every $[M,p] \in \M$ is such that $(M,\dist_M,\vol_M)$ is isomorphic with $(X,\dist_X,\vol_X)$. 
\end{lem}

\begin{proof}
%In order to `break symmetry', let $S_M \subset M$ be a Poisson point process with intensity 1 for any given path-connected mm-space $M$. What is important here is that, with probability $1$, there are no nontrivial isometries of $M$ which preserve $S_M$. For $i=1,2$ let $\tilde{\mu}_i=\Law(M_i,p_i,S_i)$. 

%This lemma is obvious when $X$ has finite volume. So we will assume $X$ has infinite volume. Let $(M_i,p_i)$ be a random mm-space with $\Law(M_i,p_i) = \mu_i$. In order to break symmetry, we let $S_i \subset M_i$ be a Poisson point process with intensity $1$. Note that, with probability 1 there are no nontrivial isometries of $M_i$ which preserve $S_i$ (because $M_i$ has infinite volume). Let $\tilde{\mu}_i=\Law(M_i,p_i,S_i)$. It suffices to show that $\tilde{\mu_1}=\tilde{\mu_2}$.

Let $\Delta \subset X$ be a Borel fundamental domain for $\Lambda$. Let $\pi:X \to \M$ be the map $\pi(p)=[X,p]$. Let $\nu = \pi_*( \frac{(\vol_X)|_\Delta}{\vol_X(\Delta)})$ be the pushforward of the normalized volume on $X$ restricted to $\Delta$. It is easy to check that $\nu$ is a unimodular measure on $\M$. This shows existence.

Now suppose that $\mu$ is as in the statement of the lemma. To be precise, $\mu \in \cM_1(\M)$ is a unimodular measure such that $\mu$-almost every $[M,p] \in \M$ is such that $(M,\dist_M,\vol_M)$ is isomorphic with $(X,\dist_X,\vol_X)$. It suffices to show that $\mu=\nu$. Let $A \subset \M$ be Borel. Suppose that $\nu(A)=0$. We will show that $\mu(A)=0$. Note that $\vol_X(\pi^{-1}(A) \cap \Delta)=0$. Since $\Delta$ is a fundamental domain of a lattice, this implies $\vol_X(\pi^{-1}(A))=0$. Define a function $f$ on $\DM$ by $f( [M,p,q] )=1$ if there is a doubly-pointed isomorphism from $(M,p,q)$ to $(X,p',q')$ and $p' \in \pi^{-1}(A) \cap \Delta, q' \in \Delta$. Let $f([M,p,q])=0$ otherwise. Because $\mu$ is unimodular,
\begin{eqnarray*}
 \vol_X(\Delta) \mu(A) &\le & \iint f([M,p,q]) ~d\vol_M(q) d\mu([M,p])\\
  &=& \iint f([M,p,q]) ~d\vol_M(p) d\mu([M,q]) = 0.
  \end{eqnarray*}
So $\mu(A)=0$. Because $A$ is arbitrary, $\mu$ is absolutely continuous to $\nu$. So there exists a nonnegative Borel function $r'$ such that $d\mu = r' d\nu$. By pulling back under $\pi$ we see that there is a non-negative Borel function $r$ on $\Delta$ such that 
$$d\mu([X,p]) = r(p) d\left(\frac{\pi_*\vol_X|_\Delta}{\vol_X(\Delta)}\right)(p).$$
%and $r(p)=r(q)$ if there is an isometry mapping $p$ to $q$.
Because $\mu$ is unimodular $d\vol_X(q)d\mu([X,p]) = d\vol_X(p)d\mu([X,q])$. Therefore
$$r(p)d\vol_X(q) d\vol_X(p) = r(q) d\vol_X(p)d\vol_X(q).$$
In particular, $r(p)=r(q)$ for a.e. $p,q \in \Delta$. This implies $\mu=\nu$ as required.
\end{proof}

Next, we determine conditions under which a sequence of mm-spaces Benjamini-Schramm-converges to the unique unimodular measure concentrated on pointed isomorphism classes of mm-spaces that are isomorphic with $X$.

\begin{defn}
Given a metric space $M$ and a subset $M' \subset M$, let $\partial M' = \overline{M'} \cap \overline{M\setminus M'}$.  For $r>0$, let $N_r(M')$ be the closed radius-$r$ neighborhood of $M'$ in $M$. 
\end{defn}

\begin{defn}
If $M$ is a path-connected metric space and $M' \subset M$ then $\covrad(M'|M)$ is the supremum over all $r>0$ such that if $\pi:\widetilde{M} \to M$ is the universal cover and $p \in \pi^{-1}(M') \subset \widetilde{M}$ then $\pi$ restricted to $B_{\widetilde{M}}(p,r)$ is an isometry onto its image.
%there exists a closed curve $\gamma:S^1 \to M'$ which is homotopically nontrivial in $M$ and the length of $\gamma$ is $r$. 
\end{defn}

\begin{lem}\label{lem:X2}
Let $X$ be a pathwise-connected mm-space with a cocompact subgroup $\Lambda<\Isom(X)$. Let $\{\Gamma_i\}_{i=1}^\infty$ be a sequence of geometric subgroups of $\Isom(X)$ and $M_i \subset X/\Gamma_i$ be a finite-volume closed subspace. Suppose 
\begin{itemize}
\item $\lim_{i\to\infty} \covrad(M_i|X/\Gamma_i)= +\infty$ and 
\item $\lim_{i\to\infty} \frac{ \vol(N_r(\partial M_i))}{\vol(M_i)} = 0$ for every $r>0$.
\end{itemize}
Then $\lim_{i\to\infty} \mu_{M_i}$ exists in $\cM_1(\M)$ and is the unique unimodular measure supported on the set of pointed isomorphism classes of mm-spaces that are isomorphic with $X$.
\end{lem}

\begin{proof}
Let $p_i \in M_i$ be uniformly random (so $\mu_{M_i} = \Law([M_i,p_i])$). The two hypotheses on $\{M_i\}_{i=1}^\infty$ imply: for every $r>0$, the probability that $B_{X/\Gamma_i}(p_i,r) \subset M_i$ tends to $1$ as $i\to\infty$ for any fixed $r>0$. Moreover, the probability that $B_{X/\Gamma_i}(p_i,r)$ is isomorphic with a ball in $X$ tends to $1$ as $i\to\infty$ (the universal cover provides the isometry). It follows that if $\mu_\infty$ is any subsequential limit point of $\{\mu_{M_i}\}_{i=1}^\infty$ then $\mu_\infty$-a.e. $[M,p]$ is such that $M$ is isomorphic with $X$. Lemma \ref{lem:closed} implies $\mu_\infty$ is unimodular and Lemma \ref{lem:X} implies $\mu_\infty$ is the {\em unique} unimodular measure supported on pointed isomorphism classes of mm-spaces that are isomorphic with $X$.

It now suffices to show that $\{\mu_{M_i}\}_{i=1}^\infty$ is precompact (so that a subsequential limit exists).  Let $D \subset X$ be a compact set such that $\Lambda D= X$. Let $p'_i \in X$ be a lift of $p_i$ under the covering map $X \to X/\Gamma_i$. Let $p''_i \in D$ be a point such that $\Lambda p'_i = \Lambda p''_i$. Note that the pointed isomorphism class of $(X,p''_i)$ depends only on $p_i$. Therefore $\Law([X,p''_i]) \in \cM_1(\M)$ is well-defined. Because $D$ is compact, $\{\Law([X,p''_i])\}_{i=1}^\infty$ is precompact. Fix $r>0$. As noted before, with probability tending to $1$ as $i\to\infty$, $B_{M_i}(p_i,r)$ is isomorphic with $B_X(p''_i,r)$. So $\{\Law([B_{M_i}(p_i,r),p_i]\}_{i=1}^\infty$ is precompact  in $\cM_1(\M)$ which implies, since $r$ is arbitrary, that $\{\mu_{M_i}\}_{i=1}^\infty$ is precompact.

%By Lemma \ref{lem:precompact} the sequence $\{\mu_{M_i}\}_{i=1}^\infty$ is precompact. 

\end{proof}

\section{Proof of Theorem \ref{thm:main-manifold}}\label{sec:main}

%For the rest of this section we assume the hypotheses of Theorem \ref{thm:main-manifold}.

%\begin{defn}
%Let $X$ be a mm-space. We say that $X$ satisfies the {\em weak tubular neighborhood theorem} if for every geometric $\Gamma < \Isom(X)$ and every simple closed curve $c$ in $X/\Gamma$ there is an $\epsilon>0$ such that the radius $\epsilon$ neighborhood of $c$ is homotopic to a circle.
%\end{defn}

We will derive Theorem \ref{thm:main-manifold} from Theorem \ref{thm:main0} below which essentially, is a version of Theorem \ref{thm:main-manifold} for metric-measure spaces. First we need a few definitions.
\begin{defn}
Let $X$ be a metric measure space and $r>0$. We define the {\em radius-$r$ Cheeger constant of $X$} by
$$h_r(X) = \inf_M \frac{ \vol_X( N_r(\partial M))}{\vol_X(M)}$$
where the infimum is over all pathwise-connected compact subsets $M \subset X$ with positive volume, $\partial M = M \cap \overline{X \setminus M}$ and $N_r(\partial M)$ is the closed radius-$r$ neighborhood of $M$ and $\vol_X(M)\le \vol_X(X)/2$.
\end{defn}

\begin{defn}
For any class of groups $\cF$, metric measure space $X$ and $r>0$ let $I_r(X|\cF)=\inf_\Gamma h_r(X/\Gamma)$ where the infimum is over all geometric $\Gamma < \Isom(X)$ such that $\Gamma$ is isomorphic to a group in $\cF$.
\end{defn}

\begin{thm}\label{thm:main0}
Let $X$ be a contractible special mm-space (Definition \ref{defn:special}). Suppose:
\begin{itemize}
\item there exists a residually finite geometric cocompact lattice $\Lambda < \Isom(X)$ and $b^{(2)}_d(\Lambda)>0$;
\item there exists an $\epsilon>0$ such that every ball of radius $\le 10\epsilon$ in $X$  is strongly convex.
%\begin{itemize}
%\item ;
%\item the volumes $\vol_X(B_X(r,p))$ vary continuously in $r>0$ and $p\in X$. %Moreover $\vol_X(B_X(r,p))>0$ for every $r>0$, $p\in X$. 
%\end{itemize}
\end{itemize}
Then there exists $r>0$ such that $I_r(X|\cG_d)>0$ where $\cG_d$ is as in Theorem \ref{thm:main-manifold}.
\end{thm}

\begin{example}
Let $d>2$, $T_d$ denote the $d$-regular tree and $X_d=T_d \times T_d$. We could consider $T_d$ to be a metric measure space by making each edge isomorphic with the unit interval. Then let $\vol_{X_d}=\vol_{T_d}\times \vol_{T_d}$ and set $\dist_{X_d}$ equal to the sum of the distances of its coordinate projections. This makes $X_d$ into a CAT(0) space and therefore every ball is strongly convex. Moreover, $\Isom(X_d)$ equals the automorphism group of $T_d \times T_d$ as a cell-complex.

Because every lattice $\Lambda < \rm{Aut}(T_d \times T_d)$ has $b^{(2)}_2(\Lambda)>0$, it follows from Theorem \ref{thm:main0} that $I_r(X_d|\cG_2)>0$ for some $r>0$. The second $L^2$-Betti numbers of free groups vanish. So $h_r(X_d/\Gamma)\ge I_r(X_d|\cG_2)>0$ for any free group $\Gamma < \Isom(X_d)$. 
\end{example}

The next lemma shows that by passing to a subgroup $\Gamma''_i <\Gamma_i$ we may substantially simplify the problem. We will need the following definition:
\begin{defn}[Asymptotic lower Betti numbers]
Let $\Gamma$ be a residually finite countable group and $d \ge 1$ an integer. Let
$$\widehat{b}_d(\Gamma) = \liminf_N \frac{b_d(N)}{[\Gamma:N]}$$
where the limit is over the net of finite-index normal subgroups of $\Gamma$ ordered by reverse inclusion. Equivalently, $\widehat{b}_d(\Gamma)$ is the smallest number $x$ such that for every $\epsilon>0$ and every finite-index normal subgroup $N \vartriangleleft \Gamma$ there exists a finite-index normal subgroup $N' \vartriangleleft \Gamma$ with $N' < N$ and 
$$\left|x- \frac{b_d(N')}{[\Gamma:N']}\right|<\epsilon.$$
In the special case that $\Gamma$ has a finite classifying space, $\widehat{b}_d(\Gamma) = b^{(2)}_d(\Gamma)$ by Theorem \ref{thm:Luck}.
\end{defn}

\begin{lem}\label{lem:injrad}
Let $X$ be as in Theorem \ref{thm:main0}. Let $\{\Gamma_i\}_{i=1}^\infty$ be a sequence of geometric residually finite subgroups $\Gamma_i < \Isom(X)$ such that $\lim_{i\to \infty} h_r(X/\Gamma_i) =0$ for every $r>0$.  

Then there exist subgroups $\Gamma''_i<\Gamma'_i<\Gamma_i$ and positive volume compact subsets $M'_i \subset X/\Gamma'_i, M''_i \subset X/\Gamma''_i$ such that 
\begin{enumerate}
%\item $\Gamma''_i$ is a finite-index subgroup of  $\Gamma'_i$ which is a subgroup of $\Gamma_i$ for each $i$;
\item $M''_i$ is a pathwise connected compact subset of $X/\Gamma''_i$ for each $i$;
\item $\lim_{i\to\infty} \frac{\vol(N_r(\partial M''_i))}{\vol(M''_i)} = 0$ for every $r$;
\item $\lim_{i\to\infty} \covrad(M''_i | X/\Gamma''_i ) = \infty$;
%\item $M'_i \subset X/\Gamma'_i$ and $M''_i$ is the pullback of $M'_i$ under the covering map $X/\Gamma''_i \to X/\Gamma'_i$;
\item 
$$\liminf_{i\to\infty} \frac{\widehat{b}_d(\Gamma'_i) }{\vol(M'_i)} = \liminf_{i\to\infty} \frac{b_d(\Gamma''_i) }{\vol(M''_i)}.$$
%\item $\lim_{i\to\infty} \frac{b_d(\Gamma''_i)}{b^{(2)}_d(\Gamma'_i)[\Gamma'_i:\Gamma''_i]} =  1.$
%\item $\liminf_{i\to\infty} \frac{[\Gamma'_i:\Gamma''_i]}{\vol(M_i)} > \eta$ for some constant $c>0$.
\end{enumerate}
\end{lem}

\begin{proof}
%Because $\lim_{i\to\infty} h(X/\Gamma_i) = 0$, Lemma \ref{lem:Buser} implies the existence of compacts $U_i \subset X/\Gamma_i$ such that 
%$$\lim_{i\to\infty} \frac{\vol(N_r(\partial U_i))}{\vol(U_i)} = 0$$
%for every $r>0$. 

By hypothesis, there exist path-connected positive volume compact sets $M_i \subset X/\Gamma_i$ such that 
$$\lim_{i\to\infty} \frac{ \vol( N_r(\partial M_i))}{\vol(M_i)}=0$$
for every $r>0$ where we have dropped the subscript on $\vol_{X/\Gamma_i}(\cdot)$ for simplicity. 

Because $X$ is contractible and $\Gamma_i$ acts freely and properly discontinuously, we may identify $\Gamma_i$ with the fundamental group $\pi_1(X/\Gamma_i)$. Let $\Gamma'_i<\Gamma_i$ be the image of $\pi_1(M_i)$ under the natural map from $\pi_1(M_i) \to \pi_1(X/\Gamma_i)$ induced by inclusion $M_i \to X/\Gamma_i$. Let $\phi_i: X/\Gamma'_i \to X/\Gamma_i$ be the covering map and let $M'_i$ be a path-connected component of $\phi_i^{-1}(M_i)$. The choice of $\Gamma'_i$ implies that $\phi_i$ restricted to $M'_i$ is a homeomorphism onto $M_i$. So $M'_i$ is compact and 
\begin{eqnarray}\label{eqn:aa}
\lim_{i\to\infty} \frac{ \vol( N_r(\partial M'_i))}{\vol(M'_i)}=0
\end{eqnarray}
for every $r>0$.

For each $\gamma \in \Gamma'_i$, let $L_i(\gamma)$ denote the infimum over all numbers $r$ such that there is a $p \in X$ whose image in $X/\Gamma'_i$ is contained in $M'_i$ and $\dist_X(p,\gamma p)\le r$. This number depends only on the conjugacy class of $\gamma$ in $\Gamma'_i$. So we may think of $L_i$ as a function on the set of conjugacy classes of $\Gamma'_i$.

 Let $r>0$. We claim that there are only a finite number of $\Gamma'_i$-conjugacy classes $[\gamma]$ with $L_i([\gamma])\le r$. To obtain a contradiction, suppose $\gamma_1,\gamma_2,\ldots \in \Gamma'_i$ is an infinite sequence of pairwise non-conjugate elements with $L_i(\gamma_i) \le r$. Let $p_i \in X$ be such that the image of $p_i$ in $X/\Gamma'_i$ is in $M'_i$ and $\dist_X(p_i,\gamma_i p_i) \le r$. Let $Y_i \subset X$ be a compact set which surjects onto $M'_i$ under the covering map $X \mapsto X/\Gamma'_i$. After conjugating $\gamma_i$ if necessary, we may assume that $p_i \in Y_i$ for all $i$. After passing to a subsequence if necessary, we may assume $\lim_{i\to\infty} p_i = p_\infty$ and $\lim_{i\to\infty} \gamma_i p_i = q_\infty$ exist. It follows that $\lim_{i,j\to\infty} \gamma_j \gamma_i^{-1} q_\infty = q_\infty$. This contradicts the assumption that $\Gamma_i$ acts properly discontinuously and freely on $X$. So there are only a finite number of $\Gamma'_i$-conjugacy classes $[\gamma]$ with $L_i([\gamma])\le r$ as claimed.
  
  Because $\Gamma_i$ is residually finite, $\Gamma'_i$ is also residually finite. So there is a finite index normal subgroup $\Gamma''_i<\Gamma'_i$ such that $\Gamma''_i$ does not contain any nontrivial element $\gamma \in \Gamma'_i$ with $L_i(\gamma)\le i$.  We may choose $\Gamma''_i$ to also satisfy
  $$\left| \frac{b_d(\Gamma''_i)}{[\Gamma'_i:\Gamma''_i]} - \widehat{b}_d(\Gamma'_i)\right| < \frac{\vol(M'_i)}{i}.$$
  This implies
\begin{eqnarray}\label{eqn:ab}
\liminf_{i\to\infty} \frac{\widehat{b}_d(\Gamma'_i) }{\vol(M'_i)} = \liminf_{i\to\infty} \frac{b_d(\Gamma''_i) }{[\Gamma'_i:\Gamma''_i] \vol(M'_i)}.
\end{eqnarray}
% $$\lim_{i\to\infty} \frac{b_d(\Gamma'_i)}{b^{(2)}_d(\Gamma_i)[\Gamma_i:\Gamma'_i]} =  1.$$
   Let $\psi_i : X/\Gamma''_i \to X/\Gamma'_i$ be the quotient map and let $M''_i = \psi_i^{-1}(M'_i)$. Because $\pi_1(M_i)$ surjects onto $\Gamma'_i$ (under the natural map from $\pi_1(M_i) \to \pi_1(X/\Gamma_i)$), it follows that $\pi_1(M'_i)$ also surjects onto $\pi_1(X/\Gamma'_i) \simeq \Gamma'_i$. This implies $M''_i$ is path-connected. 
   
      Note $\covrad(M''_i| X/\Gamma''_i ) \ge i/2$. So $\lim_{i\to\infty} \covrad(M''_i| X/\Gamma''_i )  = \infty$.
 
 The restriction of $\psi_i$ to $N_r(M''_i)$ is a finite-degree covering map onto $N_r(M'_i)$. So 
 $$\vol(N_r(\partial M''_i)) = [\Gamma'_i:\Gamma''_i] \vol(N_r(\partial M'_i)), \quad \vol(M''_i)= [\Gamma'_i:\Gamma''_i] \vol(M'_i).$$
 Now (\ref{eqn:aa},\ref{eqn:ab}) imply $\lim_{i\to\infty}  \frac{\vol(N_r(\partial M''_i))}{\vol(M''_i)} = 0$ and
 $$\liminf_{i\to\infty} \frac{\widehat{b}_d(\Gamma'_i) }{\vol(M'_i)} = \liminf_{i\to\infty} \frac{b_d(\Gamma''_i) }{\vol(M''_i)}.$$
  \end{proof}

\begin{proof}[Proof of Theorem \ref{thm:main0}]
%Observe that Theorem \ref{thm:main-manifold} follows from Lemma \ref{lem:injrad} above and Theorem \ref{thm:2}. So we need only prove Theorem \ref{thm:2}. In particular, we will assume that $\lim_{i\to\infty} \covrad(X/\Gamma_i) = \infty$.

%Let $\Lambda<\Isom(X)$ be a cocompact geometric subgroup with $b^{(2)}_d(\Lambda)>0$ for some integer $d\ge 1$.
Let $\{\Gamma_i\}_{i=1}^\infty$ be a sequence of geometric residually finite subgroups $\Gamma_i < \Isom(X)$ such that $\lim_{i\to \infty} h_r(X/\Gamma_i) =0$ for every $r>0$.   Let $\Gamma''_i<\Gamma'_i<\Gamma_i$, $M'_i \subset X/\Gamma'_i, M''_i \subset X/\Gamma''_i$ be as Lemma \ref{lem:injrad}. It suffices to show that for all but finitely $i$, $\widehat{b}_d(\Gamma'_i)>0$. 
 
  %, there exist finite-index subgroups $\Gamma'_i < \Gamma_i$ and smooth compact manifolds $M_i \subset X/\Gamma'_i$ such that 
%\begin{itemize}
%\item $\lim_{i\to\infty} \frac{\vol(N_r(\partial M_i))}{\vol(M_i)} = 0$ for every $r$;
%\item $\lim_{i\to\infty} \covrad(M_i| X/\Gamma'_i )  = \infty$.
%\end{itemize}
%If $\lim_{i\to\infty} \covrad(X/\Gamma_i) = +\infty$, then by Lemma \ref{lem:Buser} we may choose $\Gamma'_i = \Gamma_i$. It suffices to show that $\widehat{b}_d(\Gamma'_i)>0$ for all but finitely many $i$ (assuming $\widehat{b}_d(\Lambda)>0$). 

%This enables us to prove Theorem \ref{thm:main-manifold} and Theorem \ref{thm:2} simultaneously. 

%Since $h(X/\Gamma_i) \to 0$, Lemma \ref{lem:Buser} implies that there exist compact manifolds $M_i \subset X/\Gamma_i$ such that $\vol(N_r(\partial M_i))/\vol(M_i) \to 0$ for every $r>0$.

By hypothesis, there exists an $\epsilon>0$ such that every ball of radius $\le 10\epsilon$ in $X$  is strongly convex. 
For sufficiently large $i$, $ \covrad(M''_i| X/\Gamma''_i ) >10\epsilon$ which implies that each ball of radius $\le 10\epsilon$ with center in $N_{10\epsilon}(M''_i)$ is strongly convex. Moreover, for any $p \in X$ there is some $r>0$ such that $B_X(p,r)$ maps isometrically onto the ball of radius $r$ centered at the image of $p$ in $X/\Gamma''_i$. This is because $\Gamma''_i$ acts properly discontinuously and freely (because $\Gamma_i$ does and $\Gamma''_i<\Gamma_i$). So for every $q \in X/\Gamma''_i$ there is some number $\kappa(q)$ such that every ball of radius $\le \kappa(q)$ centered at $q$ is strongly convex.

Let $S_i \subset X/\Gamma''_i$ be a set and $\rho:S_i \to (0,\infty)$ be a function such that
\begin{itemize}
\item $S_i \cap M''_i$ is $\epsilon$-separated and $10\epsilon$-covers $M''_i$;
\item $\rho(s) = 10\epsilon$ for every $s \in S_i \cap M''_i$;
\item $\rho(s) \le 10\epsilon$ for all $s\in S_i$;
\item $B_{X/\Gamma''_i}(s, r )$ is strongly convex for every $s\in S_i$ and $r\le \rho(s)$;
\item $\{ B^o_{X/\Gamma''_i}(s, \rho(s)) :~s\in S_i\}$ is locally finite and covers $X/\Gamma''_i$.
\end{itemize}
Let
\begin{eqnarray*}
U_i &=& \bigcup \{ B_{X/\Gamma''_i}^o(s, \rho(s) ):~ s\in S_i \cap M''_i \}.
\end{eqnarray*}
Observe that $M''_i \subset U_i$ and $U_i$ is pathwise connected (because $M''_i$ is pathwise connected). %Let $U_i$ be the connected component of $W_i$ containing $M''_i$. Let 
%\begin{eqnarray*}
%S^U_i&=&\{s \in S_i:~B_{X/\Gamma''_i}^o(s, \rho(s) ) \subset U_i\} \\
%S^V_i&=&S_i \setminus S^U_i \\%\cup \{s\in S^U_i \cap N_{10\epsilon}(\partial M)\}\\
%V_i &=&  \bigcup \{ B_{X/\Gamma''_i}^o(s, \rho(s) ):~s\in S^V_i\}.
%\end{eqnarray*}
Because $S_i \cap M''_i$ is finite we can choose $0<\delta_i<\epsilon$ so that if 
$$U'_i:=\bigcup \{ B_{X/\Gamma''_i}(s, \rho(s) - \delta_i ):~ s\in S_i \cap M''_i\}$$
then $U'_i$ is homologically equivalent to $U_i$ (in the sense that they have the same Betti numbers), $\lim_{i\to\infty} \frac{\vol(U'_i)}{\vol(U_i)} = 1$ and $M''_i \subset U'_i$. In particular $U'_i$ is pathwise connected. Note $U'_i$ is closed while $U_i$ is open. For simplicity we have dropped the subscript on $\vol_{X/\Gamma''_i}(\cdot)=\vol(\cdot)$.

Because $M''_i \subset U'_i \subset N_{10\epsilon}(M''_i)$ it follows that \begin{itemize}
\item $\lim_{i\to\infty} \covrad(U'_i|X/\Gamma_i)= +\infty$ and 
\item $\limsup_{i\to\infty} \frac{ \vol(N_r(\partial U'_i))}{\vol(U'_i)} \le \limsup_{i\to\infty} \frac{ \vol(N_{r+10\epsilon}(\partial M''_i))}{\vol(M''_i)}= 0$ for every $r>0$.
\end{itemize}
Let $p_i$ be a uniformly random point of $U'_i$ and let $\mu_i=\Law(U'_i,p_i) \in \cM_1(\M)$.  By Lemma \ref{lem:X2}, $\lim_{i\to\infty} \mu_i = \mu_\infty$ is the unique unimodular measure supported on pointed isomorphism classes of metric measure spaces that are isomorphic with $X$.

To apply Theorem \ref{thm:Elek} (to $U'_i$) we need to check a few more hypotheses. We claim that there is a $v_0>0$ such that for every $p,q \in X$ if $\dist_X(p,q) \le 10\epsilon-\delta_i$ then $\vol(B_X(q,\epsilon/2) \cap B_X(p,10\epsilon-\delta_i)) >v_0$. If this is false then there are sequences $\{p_j\}_{j=1}^\infty, \{q_j\}_{j=1}^\infty \subset X$ and $\{i_j\}_{j=1}^\infty \subset \N$ such that $\dist_X(p_j,q_j) \le 10\epsilon-\delta_{i_j}$ and $\lim_{j\to\infty} \vol(B_X(q_j,\epsilon/2) \cap B_X(p_j,10\epsilon-\delta_{i_j})) =0$. Let $D \subset X$ be a compact set that surjects onto $X/\Lambda$. By replacing $p_j,q_j$ with $g_jp_j, g_jq_j$ for some $g_j \in \Lambda$ if necessary, we may assume that each $p_j \in D$. After passing to a subsequence if necessary, we may assume $\lim_{j\to\infty} p_j = p_\infty, \lim_{j\to\infty} q_j = q_\infty$ and $\lim_{j\to\infty} \delta_{i_j} =\delta_\infty \in [0,\epsilon]$ exist. Let $\eta>0$. For all sufficiently large $j$,
$$B_X(q_\infty,\epsilon/2-\eta) \cap B_X(p_\infty,10\epsilon-\delta_\infty-\eta) \subset B_X(q_j,\epsilon/2) \cap B_X(p_j,10\epsilon-\delta_{i_j}).$$
This implies $\vol_X( B^o_X(q_\infty,\epsilon/2) \cap B^o_X(p_\infty,10\epsilon-\delta_\infty) )=0$. However, $B_X(p_\infty, 10\epsilon-\delta_\infty)$ is strongly convex and so there is a geodesic from $p_\infty$ to $q_\infty$ in $B_X(p_\infty, 10\epsilon-\delta_\infty)$. It follows that $B^o_X(q_\infty,\epsilon/2) \cap B^o_X(p_\infty,10\epsilon-\delta_\infty)$ is a nonempty open set. Since $\vol_X$ is fully supported (because $X$ is special), this is a contradiction. This proves the claim. Note that $v_0$ does not depend on $i$.

If $i$ is sufficiently large, then $\covrad(U'_i| X/\Gamma''_i) > 10\epsilon$ which implies that every $(10\epsilon-\delta_i)$-ball in $X/\Gamma''_i$ which lies in $U'_i$ is isometric with a $(10\epsilon-\delta_i)$-ball in $X$. Therefore, for every $q_i \in U'_i$, $\vol( B_{X/\Gamma''_i}(q_i,\epsilon/2) \cap U'_i) = \vol(B_{U'_i}(q_i,\epsilon/2))> v_0$. Also because $X/\Lambda$ is compact, there is a $v_1>0$ such that $B_X(x,20\epsilon)<v_1$ for every $x\in X$. This implies $B_{U'_i}(q_i,20\epsilon)<v_1$ too. The hypotheses of Theorem \ref{thm:Elek} have now been checked. That result implies $\lim_{i\to\infty} \frac{b_d(U'_i)}{\vol(U'_i)}$ exists. 

%There is a constant $\Delta>0$ such that the maximum cardinality of an $\epsilon$-separated subset of $B_X(p,20\epsilon)$ is at most $\Delta$ for any $p \in X$ (because $X/\Lambda$ is compact). This implies the same bound on $X/\Gamma''_i$ and therefore the same bound on $U'_i$. 

Because $\Lambda$ is residually finite, there exists a decreasing sequence $\{\Lambda_i\}_{i=1}^\infty$ of finite-index normal subgroups of $\Lambda$ such that $\cap_{i=1}^\infty \Lambda_i = \{e\}$. Note that the covering radius of $X/\Lambda_i$ tends to infinity as $i\to\infty$. So Lemma \ref{lem:X2} implies $\lim_{i\to\infty} \mu_{X/\Lambda_i}  = \mu_\infty$. Theorem \ref{thm:Elek} now implies
$$\lim_{i\to\infty} \frac{b_d(U'_i)}{\vol(U'_i)} = \lim_{i\to\infty} \frac{b_d(X/\Lambda_i) }{\vol(X/\Lambda_i)}.$$
Because $X$ is contractible, $X/\Lambda_i$ is a classifying space for $\Lambda_i$, which implies $b_d(X/\Lambda_i) = b_d(\Lambda_i)$. Because $X/\Lambda_i$ is a $[\Lambda:\Lambda_i]$-fold cover of $X/\Lambda$, it follows that $\vol(X/\Lambda_i) = [\Lambda:\Lambda_i]\vol(X/\Lambda)$.  By Theorem \ref{thm:Luck}.
$$\lim_{i\to\infty} \frac{b_d(X/\Lambda_i) }{\vol(X/\Lambda_i)} = \lim_{i\to\infty} \frac{b_d(\Lambda_i) }{[\Lambda:\Lambda_i]\vol(X/\Lambda)} = \frac{b^{(2)}_d(\Lambda)}{\vol(X/\Lambda)}.$$
So we have established:
\begin{eqnarray*}%\label{eqn:1}
\lim_{i\to\infty} \frac{b_d(U'_i)}{\vol(U'_i)} = \frac{b^{(2)}_d(\Lambda)}{\vol(X/\Lambda)}.
\end{eqnarray*}
Because $U'_i$ is homologically equivalent to $U_i$ and $\lim_{i\to\infty} \frac{\vol(U'_i)}{\vol(U_i)} = 1$, we have
\begin{eqnarray}\label{eqn:1}
\lim_{i\to\infty} \frac{b_d(U_i)}{\vol(U_i)} = \frac{b^{(2)}_d(\Lambda)}{\vol(X/\Lambda)}.
\end{eqnarray}

Let 
\begin{eqnarray*}
W_i &=& \bigcup \{ B_{X/\Gamma''_i}^o(s, \rho(s) ):~s\in S_i \setminus M''_i \}\\
S^V_i &=& (S_i \setminus M''_i) \cup \{ s\in S_i \cap M''_i:~ B^o_{X/\Gamma''_i}(s,\rho(s)) \cap W_i \ne \emptyset\} \\
V_i &=& \bigcup \{ B_{X/\Gamma''_i}^o(s, \rho(s) ):~s\in S^V_i \}.
\end{eqnarray*}

Let $K_i$ be the nerve complex of $\{B^o_{X/\Gamma''_i}(s, \rho(s)):~s\in S_i\}$. Let $K^U_i \subset K_i$ be the nerve complex of $\{B^o_{X/\Gamma''_i}(s,\rho(s)):~s\in S_i \cap M''_i\}$. Similarly, let $K^V_i \subset K_i$ be the nerve complex of $\{B^o_{X/\Gamma''_i}(s, \rho(s) ):~s\in S^V_i \}$.

Because each $B^o_{X/\Gamma''_i}(s,\rho(s))$ is strongly convex (for $s\in S_i$), it follows that any nonempty intersection of such balls is also strongly convex and is therefore contractible \cite{Ro70}. By \cite[Corollary 4G.3]{Ha02}, this implies $K_i$ is homotopic to $X/\Gamma''_i$, $K^U_i$ is homotopic to $U_i$ and $K^V_i$ is homotopic to $V_i$. Therefore, $b_d(K_i) = b_d(X/\Gamma''_i) = b_d(\Gamma''_i)$ (since $X/\Gamma''_i$ is a classifying space for $\Gamma''_i$ since $X$ is contractible), $b_d(K^U_i)=b_d(U_i)$ and $b_d(K^V_i)=b_d(V_i)$.

%Moreover,  $K^U_i \cap K^V_i$ is the nerve complex of $\{B^o_{X/\Gamma''_i}(s,\rho(s)):~s\in S_i \cap M''_i, B^o_{X/\Gamma''_i}(s,\rho(s)) \cap \partial M''_i \ne \emptyset\}$. So $K^U_i \cap K^V_i$ is homotopic to
%$$W_i := \bigcup \{ B_{X/\Gamma''_i}^o(s, \rho(s) ):~ s\in S_i \cap M''_i, B^o_{X/\Gamma''_i}(s,\rho(s)) \cap \partial M''_i \ne \emptyset\}.$$
%We claim that $W_i = U_i \cap V_i$. It is immediate that $W_i \subset U_i \cap V_i$. If $x \in U_i \cap V_i$ then $x \in B^o_{X/\Gamma''_i}(s,\rho(s))$ for some $s \in S_i \cap M''_i$. If $B^o_{X/\Gamma''_i}(s,\rho(s)) \cap \partial M''_i \ne \emptyset$ then $x\in W_i$. If this is not true then  In particular, $x \in N_{10\epsilon}(M''_i)$. 

% and $b_d(K^U_i \cap K^V_i) = b_d(W_i)$. 

We claim that $K^U_i \cup K^V_i = K_i$. To see this suppose $T \subset S_i$ spans a simplex in $K_i$. Then either $T \subset K^U_i$ or there exists $s\in T \setminus M''_i$. For any $t \in T$, $B_{X/\Gamma''_i}^o(t,\rho(t)) \cap B_{X/\Gamma''_i}^o(s,\rho(s))  \ne \emptyset$. Since $B_{X/\Gamma''_i}^o(s,\rho(s)) \subset W_i$, this implies $t \in S^V_i$. Since $t$ is arbitrary, the simplex spanning $T$ is contained in $K^V_i$. Since $T$ is arbitrary, $K^U_i \cup K^V_i = K_i$.

The Mayer-Vietoris sequence
$$\cdots \to H_d(K^U_i \cap K^V_i) \to H_d(K^U_i) \oplus H_d(K^V_i) \to H_d(K_i) \to \cdots$$
implies
\begin{eqnarray}\label{eqn:2}
b_d(U_i) = b_d(K^U_i) \le b_d(K_i) + b_d(K^U_i \cap K^V_i) =  b_d(\Gamma''_i) + b_d(K^U_i \cap K^V_i).
\end{eqnarray}
If $s \in M''_i$ and $Z \subset B_{M''_i}(s,20\epsilon)$ is any $\epsilon$-separated subset then because 
$$v_1>\vol(B_{M''_i}^o(q,20\epsilon)) \ge \vol(B^o_{M''_i}(q,\epsilon/2)) > v_0>0$$
for every $q\in M''_i$, we must have $v_0|Z| \le v_1$. So $|Z| \le v_1/v_0$. So setting $\Delta:=v_1/v_0$, we see that the degree of any vertex of $K^U_i$ is at most $\Delta$. So $b_d(K^U_i \cap K^V_i)$ is at most the number of $d$-simplices in $K^U_i \cap K^V_i$ which is at most the number of vertices of $K^U_i \cap K^V_i$ multiplied by ${\Delta \choose d}$. The vertex set of $K^U_i \cap K^V_i$ is $S'_i= \{ s\in S_i \cap M''_i:~ B^o_{X/\Gamma''_i}(s,\rho(s)) \cap W_i \ne \emptyset\}$.  So
$$b_d(K^U_i \cap K^V_i) \le |S'_i| {\Delta \choose d}.$$

Note $S'_i$ is contained in the $20\epsilon$-neighborhood of $\partial M''_i$. Because $S'_i$ is $\epsilon$-separated and each $(\epsilon/2)$-ball has volume at least $v_0$ (for some $v_0>0$ independent of $i$), we have $|S'_i| v_0 \le \vol(N_{20\epsilon}(\partial M''_i))$. So
$$b_d(K^U_i \cap K^V_i) \le v_0^{-1} \vol(N_{20\epsilon}(\partial M''_i)) {\Delta \choose d}.$$
Therefore,
\begin{eqnarray*}
\limsup_{i\to\infty} \frac{b_d(K^U_i \cap K^V_i)}{\vol(U_i)}&\le&  {\Delta \choose d} v_0^{-1} \limsup_{i\to\infty} \frac{\vol(N_{20\epsilon}(\partial M''_i))}{\vol(U_i)}\\
&\le & {\Delta \choose d} v_0^{-1}\limsup_{i\to\infty} \frac{\vol(N_{20\epsilon}(\partial M''_i))}{\vol(M''_i)} = 0.
\end{eqnarray*}
%Because $U'_i$ is homotopic to $U_i$, equation (\ref{eqn:2}) now implies
%$$\liminf_{i\to\infty} \frac{b_d(U'_i)}{\vol(U_i)} \le \liminf_{i\to\infty} \frac{b_d(\Gamma''_i)}{\vol(U_i)}.$$
Lemma \ref{lem:injrad}, the fact that $M''_i \subset U_i$ and equations (\ref{eqn:1},\ref{eqn:2}) and now imply
\begin{eqnarray*}\label{eqn:3}
\liminf_{i\to\infty} \frac{\widehat{b}_d(\Gamma'_i)}{\vol(M'_i)}& =&\liminf_{i\to\infty} \frac{b_d(\Gamma''_i)}{\vol(M''_i)} \ge  \liminf_{i\to\infty} \frac{b_d(\Gamma''_i)}{\vol(U_i)} \\
&\ge& \liminf_{i\to\infty} \frac{b_d(U_i)}{\vol(U_i)}  -  \frac{b_d(K^U_i \cap K^V_i)}{\vol(U_i)}=  \frac{b^{(2)}_d(\Lambda)}{\vol(X/\Lambda)}>0.
\end{eqnarray*}
So $\widehat{b}_d(\Gamma'_i)>0$ for all but finitely many $i$. This implies the theorem.
\end{proof}

We now turn to the proof of Theorem \ref{thm:main-manifold}. We will need the following lemma to smooth out the Cheeger submanifolds of $X/\Gamma$.
 
\begin{lem}[Hair-cutting Lemma]\label{lem:Buser}
Let $M$ be an infinite volume complete Riemannian $n$-manifold. Suppose there is a $\delta>0 $ such that the Ricci curvature of $M$ is at least $-\delta^2(n-1)$ (everywhere). Suppose as well that $h(M) < 1$. Then there exist a pathwise connected compact subset  $M'' \subset M$ and a function $f:\R_{>0} \to \R_{>0}$ such that for every $R>0$
\begin{eqnarray}\label{eqn:Buser}
 \frac{\textrm{vol}(N_R(\partial M''))}{\textrm{vol}(M'')} \le f(R) h(M).
 \end{eqnarray}
%where $N_R(\partial M'') = \{x \in M:~ \dist_M(x,\partial M'') \le R\}$. 
Moreover $f$ depends only on $\delta$ and $\textrm{dim}(M)$.
\end{lem}

\begin{proof}
This is contained in Lemma 7.2 of \cite{Bu82} except in one detail: $M''$ is not required to be pathwise connected. However, a small perturbation of the proof yields a pathwise connected subset. To explain this, let us recall the construction of $M'$ from \cite{Bu82}. Let $\epsilon>0$ and $A$ be a smooth compact submanifold of $M$ with
$$\frac{\area(\partial A)}{\vol(A)} \le h(M)(1+\epsilon).$$
%Without loss of generality we may assume $A$ is pathwise connected.

Let $r>0$ be a sufficiently small constant (how small depends only on the dimension). Let
$$M' = \{ p \in M:~ \vol(A \cap B_M(p,r)) > (1/2)\vol (B_M(p,r))\}.$$
Note that $\partial M' = \{ p \in M:~ \vol(A \cap B_M(p,r)) = (1/2)\vol (B_M(p,r))\}$.

[In Buser's notation, $B_M(p,r)$ is denoted by $U(p,r)$, $M'$ is denoted by $\widetilde{A}$, $\partial M'$ is denoted by $\widetilde{X}$, $N_t(\partial M')$ is denoted by $\widetilde{X}^t$, $\frac{\area(\partial A)}{\vol(A)}$ is denoted by $\mathscr{H}$.]

Let $K_1,\ldots, K_m$ be the components of $M'$. Observe that
$$\frac{\area(\partial A)}{\vol(A \cap M')} = \sum_{i=1}^m \frac{\area(\partial A \cap K_i)}{\vol(K_i \cap A)} \frac{\vol(K_i \cap A)}{\vol(A \cap M')}.$$
In particular, $\frac{\area(\partial A)}{\vol(A \cap M')}$ is a convex sum of $\frac{\area(\partial A \cap K_i)}{\vol(K_i \cap A)}$. So there exists a component $K_i$ such that
$$\frac{\area(\partial A \cap K_i)}{\vol(K_i \cap A)} \le \frac{\area(\partial A)}{\vol(A \cap M')}  \le (1+\epsilon)h(M) \frac{\vol(A)}{\vol(A \cap M')}.$$
According to \cite[equations 4.6, 4.9]{Bu82}, $\vol(A \cap M') \ge c\vol(A)$ where $c=1-\frac{4 \mathscr{H}\beta(4r)}{j(r)\beta(r)} \ge 1/2$ (in Buser's notation). Therefore,
$$\frac{\area(\partial A \cap K_i)}{\vol(K_i \cap A)} \le \frac{\area(\partial A)}{\vol(A \cap M')}  \le 2(1+\epsilon)h(M).$$

Let $M''$ be the closure of $K_i$. It is now possible to replace $M'$ with $M''$ in the proof of \cite[Lemma 7.2]{Bu82} (which is mostly contained in \S 4 of \cite{Bu82}) to conclude that $M''$ satisfies (\ref{eqn:Buser}).

\end{proof}

\begin{proof}[Proof of Theorem \ref{thm:main-manifold}]
Because $X/\Lambda$ is compact, \cite[Theorem 7.9]{Ch93} implies that there exists an $\epsilon>0$ such that every ball of radius $\le 10\epsilon$ in $X$  is strongly convex. So Theorem \ref{thm:main0} implies $I_r(X|\cG_d)>0$ for some $r>0$. Lemma \ref{lem:Buser} implies that if $I(X|\cG_d)<1$ then $I_r(X|\cG_d) \le f(r)I(X|\cG_d)$ for some function $f$ which depends only on the dimension of $X$ and a lower bound on its Ricci curvature. Thus $I(X|\cG_d)>0$.
\end{proof}

\section{Applications}\label{sec:app}
In this section we prove Corollary \ref{cor:app1}. The starting point is:

\begin{lem}\label{lem:Dodzuik}
If $\Lambda$ is a lattice in $\Isom(\H^{2n})$ for some $n\ge 1$, then $b_n^{(2)}(\Lambda)>0$. 
\end{lem}

\begin{proof}
This is contained in \cite[Theorem 5.12]{Lu02}. %(which is partly based on \cite{Do79}).
%This is implied by the main result of \cite{Do79} (which shows that $b_k^{(2)}(\Lambda)=0$ if $k \ne n$), the fact that the alternating sum of the $L^2$ Betti numbers is the Euler characteristic of $X/\Lambda$ and the well-known fact that this is nonzero (because it is proportional to the volume of $X/\Lambda$). 
\end{proof}

\begin{remark}
\cite[Theorem 5.12]{Lu02} also shows that if $\Lambda < \Isom(\H^{n})$ is a lattice then $b^{(2)}_d(\Lambda)=0$ unless $d=n/2$ is an integer.
\end{remark}

It now suffices to show:
\begin{prop}\label{prop:3m}
If $\Gamma$ is a torsion-free lattice in $\Isom(\H^3)$ then $\Gamma \in \cG_d$ for all $d>1$.
\end{prop}

\begin{proof}
The fact that $\Gamma$ is residually finite is well-known: $\Gamma$ is linear (since it is a subgroup of $SO(3,1)$) and all finitely generated linear groups are residually finite by \cite{Ma40}. Let $\Gamma' < \Gamma$ be finitely generated. Observe that $\Gamma'$ is the fundamental group of a hyperbolic 3-manifold (namely $\H^3/\Gamma'$). By the Scott Core Theorem \cite{Sc73}, $\Gamma'$ has a finite classifying space. By L\"uck's approximation Theorem \ref{thm:Luck}, it suffices to show that $b^{(2)}_d(\Gamma')=0$ for all $d>1$. This is handled in Lemma \ref{lem:1-dim} below. In fact, we will prove something stronger: that $\Gamma$ is almost treeable, as defined next.
\end{proof}

%$b^{(2)}_k(\Gamma')=0$ for every $k\ge 2$ and $\Gamma'<\Gamma$. We will prove something stronger: that $\Gamma$ is almost treeable, as defined next.

\begin{defn}[Treeability and almost treeability]
Let $\Gamma$ be a countable discrete group. Let ${\Gamma \choose 2}$ be the set of all unordered pairs of elements in $\Gamma$ and let $\cG(\Gamma) = 2^{{\Gamma \choose 2}}$ be the set of all subsets of ${\Gamma \choose 2}$ with the product topology. Because $\Gamma$ is countable, this means that $\cG(\Gamma)$ is a compact metrizable space (in fact, it is homeomorphic to a Cantor set). Associated to any element $x \in \cG(\Gamma)$ is a graph $G_x$ with vertex set $\Gamma$ and edge set $x$.  Observe that $\Gamma$ acts on $\cG(\Gamma)$ by $g x = \{ \{ga, gb\}:~ \{a,b\} \in x\}$ for $g \in \Gamma, x\in \cG(\Gamma)$. 

Let $\cF(\Gamma)$ denote the set of all $x \in \cG(\Gamma)$ such that $G_x$ is a forest (i.e., every connected component of $G_x$ is simply connected). Let $\cT(\Gamma) \subset \cF(\Gamma)$ denote the set of all $x \in \cF(\Gamma)$ such that $G_x$ is a tree. The action of $\Gamma$ preserves both $\cF(\Gamma)$ and $\cT(\Gamma)$.

We say $\Gamma$ is {\em treeable} if there is a $\Gamma$-invariant Borel probability measure on $\cT(\Gamma)$. The group $\Gamma$ is {\em almost treeable} if for every finite set $F \subset \Gamma$ and every $\epsilon>0$ there exists a $\Gamma$-invariant Borel probability measure $\mu$ on $\cF(\Gamma)$ such that if $x \in \cF(\Gamma)$ is random with law $\mu$ then with probability $\ge 1-\epsilon$ the set $F$ is contained in a connected component of $G_x$. In particular, if $\Gamma$ is treeable then $\Gamma$ is almost treeable.
\end{defn}

Treeability was introduced in \cite{Ad88} and almost treeability first appeared in \cite{Ga05}. The connection between almost treeability and $L^2$-Betti numbers is furnished by: %In the language of \cite{Ga02}, treeability is the same as ergodic dimension 1 and almost treeability is the same as approximate dimension 1. 
\begin{lem}\label{lem:Gab}
If $\Gamma$ is almost treeable then $b_k^{(2)}(\Gamma)=0$ for every $k\ge 2$.
\end{lem}
\begin{proof}
This is  \cite[Theorem 0.8]{Ga05}.
\end{proof}

It is technically easier to work in the realm of equivalence relations. So we introduce the following definitions. 
\begin{defn}
Let $(X,\mu)$ be a standard Borel probability space and $E \subset X \times X$ a discrete Borel equivalence relation (discrete means that every equivalence class is at most countable). We say that $E$ is {\em treeable} (mod $\mu$) if there exists a Borel subset $H \subset E$ such that $H$ is symmetric (so $(a,b) \in H \Rightarrow (b,a) \in H$) and the graph $G_H$ with vertex set $X$ and edge set $\{ \{a,b\}:~(a,b) \in H\}$ is such that for $\mu$-a.e. $x\in X$ the connected component of $G_H$ containing $x$ is a tree spanning the $E$-class of $x$. 

We say that $E$ is {\em almost treeable} (mod $\mu$) if there is a sequence $\{H_i\}_{i=1}^\infty$ of symmetric Borel subsets $H_i \subset E$ such that the corresponding graphs $G_{H_i}$ are forests and for a.e. $x \in X$ and any $y$ in the $E$-class of $x$ we have that $x$ and $y$ are contained in the same component of $H_i$ for all but finitely many $i$.
\end{defn}
%\begin{remark}
%In \cite[Section 5]{Ga02} a more general view is taken. What is called treeable here is the same thing as ergodic dimension 1. What is called almost treeable here is the same thing as approximate dimension 1.
%\end{remark}

The connection between equivalence relations and groups is given by:
\begin{prop}\label{prop:eq-rel}
A group $\Gamma$ is treeable if and only if there is a free pmp (probability-measure-preserving) action $\Gamma \cc (X,\mu)$ such that if $E$ is the orbit-equivalence relation $E=\{(x,gx):~x\in X,g\in \Gamma\}$ then $E$ is treeable (mod $\mu$). Similarly, $\Gamma$ is almost treeable if and only if there is a free pmp action $\Gamma \cc (X,\mu)$ such that the orbit-equivalence relation $E$ is almost treeable (mod $\mu$).
\end{prop}

\begin{proof}
%This is an exercise in understanding the definitions. 
In the case of treeability, this is \cite[Proposition 30.1]{KM04}. The almost treeable case is similar (and an easy exercise).
\end{proof}

\begin{lem}\label{lem:sub}
If $\Gamma$ is treeable and $\Gamma' < \Gamma$ then $\Gamma'$ is treeable. Similarly if $\Gamma$ is almost treeable and $\Gamma'<\Gamma$ then $\Gamma'$ is almost treeable.
\end{lem}

\begin{proof}
%In the case of treeability, this is \cite[Ch III, Proposition 30.7]{KM04}. 
This is a consequence of Proposition \ref{prop:eq-rel} and \cite[Propositions 5.8 and 5.16]{Ga02}.
\end{proof}

\begin{lem}\label{lem:me}
Treeability and almost treeability are measure-equivalence invariants. Therefore, if $\Gamma_1,\Gamma_2$ are lattices in a locally compact group $G$ and $\Gamma_1$ is almost treeable, then $\Gamma_2$ is almost treeable.
\end{lem}

\begin{proof}
In the case of treeability this is \cite[Proposition 6.5]{Ga02}. Almost treeability is similar.
\end{proof}

\begin{lem}\label{lem:treeable}
If $\Gamma$ is the fundamental group of a surface then $\Gamma$ is treeable.
\end{lem}

\begin{proof}
If $\Gamma$ is free then this is obvious as the usual Cayley graph of $\Gamma$ is a tree. If $\Gamma$ is amenable then this is a well-known consequence of the fact there is a unique hyperfinite $II_1$-equivalence relation \cite{OW80} (see also \cite[Ch III, Proposition 30.1]{KM04} to see the connection). If $\Gamma$ is the fundamental group of a closed surface of genus $\ge 2$ then $\Gamma$ is measure-equivalent to a free group since $\Gamma$ can be realized as a lattice in $\Isom(\H^2)$ (and so can any finite rank nonamenable free group). Lemma \ref{lem:me} now implies  $\Gamma$ is treeable.
\end{proof}

\begin{lem}\label{lem:almost-treeable}
Lattices in $\Isom(\H^3)$ are almost treeable.
\end{lem}

\begin{proof}
Let $\Lambda < \Isom(\H^3)$ be a lattice such that $\H^3/\Lambda$ is a manifold which fibers over a circle with fiber a noncompact surface. It is well-known that such lattices exist (see e.g., \cite{Jo77}). Note that $\Lambda$ can be expressed as $\Lambda = F_r \rtimes_\theta \Z$ where $F_r$ denotes the free group of some rank $r \ge 2$ and $\theta:F_r \to F_r$ is an automorphism. We can therefore write elements of $\Lambda$ as pairs $(f, n)$ with $f\in F_r, n \in \Z$ subject to the multiplication rule
$$(f,n)(g,m) = (f \theta^n (g), n+m).$$
Now let $p>0$ be an integer and let $i$ be a uniformly random integer in $\{0,\ldots, p-1\}$. 

Let $S=\{s_1,\ldots, s_r\} \subset F_r$ be a free generating set. Let $E_i \in \cG(\Lambda)$ be the set containing 
\begin{itemize}
\item $\{ (f, m), (fs_j,m) \}$ for every $f \in F_r, 1\le j \le r$ and $m \in \Z$ with $p \mid (m -i)$;
\item $\{ (f,m), (f, m+1)\}$ for every $f\in F_r$ and $m \in \Z$ with $p \nmid (m-i-1)$.
\end{itemize}
Observe that the graph with vertex set $\Lambda$ and edge set $E_i$ is a forest. Moreover, the law of $E_i$ is an invariant probability measure $\lambda_p$ on $\cG(\Lambda)$. Finally, for any $(f,n),(g,m) \in \Lambda$ with $n\le m$, $(f,n), (g,m)$ are in the same connected component of $(\Lambda,E_i)$ if and only if there does not exist an integer $q$ with $n \le q < m$ such that $p \nmid (q-i-1)$. This occurs with probability equal to $\frac{p - |m-n|}{p}$ if $|m-n| \le p$. In particular, this probability tends to $1$ as $p\to\infty$. This implies $\Lambda$ is almost treeable. 

%$$\Lambda = \langle s_1,\ldots, s_n,t | t s_i t^{-1} = \theta(s_i)\rangle$$
%where 
By Lemma \ref{lem:me}, it follows that every lattice in $\Isom(\H^3)$ is almost treeable.
\end{proof}

%\begin{question}
%Do lattices in $\Isom(\H^4)$ have ergodic dimension 2 (as defined in \cite[Section 5]{Ga02})?
%\end{question}

\begin{lem}\label{lem:1-dim}
If $\Gamma'$ is a subgroup of the fundamental group $\Gamma$ of a complete finite-volume hyperbolic $3$-manifold then $b^{(2)}_d(\Gamma')=0$ for every $d\ge 2$.
\end{lem}

\begin{proof}
This is true because $\Gamma$ is almost treeable by Lemma \ref{lem:almost-treeable}, every subgroup of an almost treeable group is almost treeable by Lemma \ref{lem:sub} and any almost treeable group $\Lambda$ has $b_d^{(2)}(\Lambda)=0$ for every $d \ge 2$ by Lemma \ref{lem:Gab}.
\end{proof}

\begin{proof}[Proof of Corollary \ref{cor:app1}]
This follows from Theorem \ref{thm:main-manifold} and Proposition \ref{prop:3m}.
\end{proof}

%\section{An open problem}

%Given a group $\Gamma$ and an integer $n\ge 0$, we may ask: what is the range of $h(\H^n/ \phi(\Gamma))$ over all monomorphisms $\phi:\Gamma \to \Isom(\H^n)$ such that the image $\phi(\Gamma)$ is discrete and properly discontinuous? The same question can be asked with the Cheeger constant replaced by the Hausdorff dimension of the radial limit set, the lowest eigenvalue of the Laplace operator or the critical exponent of the Poincar\'e series for example. Likewise we may wish to impose additional restraints on the image $\phi(\Gamma)$ such as convex cocompactness or geometrical finiteness. Theorem \ref{thm:main-manifold} indicates that the $L^2$-Betti numbers of $\Gamma$ may play an important role in answering these questions. Are there other invariants (e.g., $L^2$ torsion) which can be used to limit the range?

\appendix

\section{Pointed subsets and measures of a metric space}

%To begin, we review the Prohorov metric using \cite{GPW09} as a guide.
 
The purpose of this appendix is to prove Theorem \ref{thm:Mn}. We begin by studying pointed measures and pointed subspaces of a given metric space $Z$ and their limits.
 
% \begin{defn}%[Support]
%The {\em support} of a measure $\mu$ on a space $X$ is the smallest closed subset $\supp(\mu) \subset X$ such that $\mu(X \setminus \supp(\mu)) = 0$.
%\end{defn}

 \begin{defn}
 A {\em pointed measure} on a topological space $Z$ is a pair $(\mu,p)$ where $p\in Z$ and $\mu$ is a Borel measure on $Z$. A {\em pointed subset} of $Z$ is a pair $(X,p)$ where  $X \subset Z$ and $p\in Z$.
 \end{defn}
 
 \begin{defn}
 Given a subset $F$ of a metric space $Z$, let $N_Z^o(F,\epsilon)$ denote the open $\epsilon$-neighborhood of $F$ in $Z$.
 \end{defn}
 
 \begin{defn}\label{defn:pointed-Hausdorff}
 We say that two pointed measures $(\mu_1,p_1),(\mu_2,p_2)$ on a metric space $Z$ are {\em $(\epsilon,R)$-related} if for every closed $F_i \subset B_Z(p_i,R)$,
 \begin{eqnarray*}
 \mu_1(F_1) < \mu_2 (N_Z^o(F_1,\epsilon))+\epsilon, \quad \mu_2(F_2) < \mu_1(N_Z^o(F_2,\epsilon))+\epsilon
 \end{eqnarray*}
 and $\dist_Z(p_1,p_2)<\epsilon$. We say two pointed subsets $(X_1,p_1),(X_2,p_2)$ of $Z$ are {\em $(\epsilon,R)$-related} if $\dist_Z(p_1,p_2)<\epsilon$ and 
$$B_Z( p_1, R) \cap X_1 \subset N_Z^o(  X_2, \epsilon), \quad B_Z( p_2, R) \cap X_2 \subset N_Z^o(  X_1, \epsilon).$$
A sequence $\{(X_i,p_i)\}_{i=1}^\infty$ of pointed closed subsets of $Z$ converges to $(X_\infty,p_\infty)$ in the {\em pointed Hausdorff topology} if for every $\epsilon,R>0$, there is an $I$ such that $i>I$ implies $(X_i,p_i)$ and $(X_\infty,p_\infty)$ are $(\epsilon,R)$-related.
\end{defn}

 \begin{lem}\label{lem:related}
 If pointed measures $(\mu_1,p_1),(\mu_2,p_2)$ are $(\epsilon_1,R_1)$-related and $(\mu_2,p_2),(\mu_3,p_3)$ are $(\epsilon_2,R_2)$-related then $(\mu_1,p_1), (\mu_3,p_3)$ are $(\epsilon_1+\epsilon_2, R_3)$-related where $R_3=\min\{R_1-2\epsilon_2, R_2-2\epsilon_1\}$. Similarly, if $(X_1,p_1),(X_2,p_2)$ are $(\epsilon_1,R_1)$-related pointed subsets and $(X_2,p_2), (X_3,p_3)$ are $(\epsilon_2,R_2)$-related pointed subsets then $(X_1,p_1),(X_3,p_3)$ are $(\epsilon_1+\epsilon_2, R_3)$-related.% where $R'_3=\min\{R_1-\epsilon_2, R_2-\epsilon_1\}$.  
 \end{lem}
 
 \begin{proof}
 Let $F \subset B_Z(p_1,R_3)\subset B_Z(p_1,R_1)$ be closed. Then 
 $$N_Z^o(F,\epsilon_1) \subset B_Z(p_1,R_3+\epsilon_1)  \subset B_Z(p_2,R_3+2\epsilon_1) \subset B_Z(p_2,R_2).$$
 Therefore,
  \begin{eqnarray*}
\mu_1(F) &<& \mu_2(N_Z^o(F,\epsilon_1))+\epsilon_1  < \mu_3(N_Z^o( N_Z^o(F,\epsilon_1), \epsilon_2)) +\epsilon_1+ \epsilon_2\\
&\le& \mu_3( N_Z^o(F,\epsilon_1+\epsilon_2))+\epsilon_1+\epsilon_2.
 \end{eqnarray*}
 The other inequality is similar. The result for pointed subsets is similar.
 \end{proof}

\begin{lem}\label{lem:key-Mn}
Let $Z$ be a proper metric space. Let $(\mu_i,p_i)$ (for $1\le i \le \infty$) be pointed Radon measures of $Z$ with $\lim_{i\to\infty} p_i = p_\infty$. Then $\lim_{i\to\infty} \mu_i = \mu_\infty$ in the weak* topology if and only if for every $\epsilon,R>0$ there exists $I$ such that $i>I$ implies $(\mu_i,p_i)$ and $(\mu_\infty,p_\infty)$ are $(\epsilon,R)$-related. 
\end{lem}

\begin{proof}
Suppose $\lim_{i\to\infty} \mu_i = \mu_\infty$ in the weak* topology. Let $\epsilon,R>0$. Let $\cF \subset C_c(Z)$ be a finite set such that for every compact subset $F \subset B_Z(p_\infty,R+\epsilon)$ there exists $g\in \cF$ such that $g=1$ on $F$, $g=0$ on the complement of $N_Z^o(F,\epsilon)$ and $0\le g\le 1$ on all of $Z$. To see that such a set exists, let $\cO$ be any finite open cover of $B_Z(p_\infty,R+\epsilon)$ by open balls of radius $<\epsilon$. Let $\cF' = \{g_U:~U \in \cO\}$ be a partition of unity subordinate to $\cO$. Let $\cF$ be the set of all sums of the form $\sum\{ g_U:~ U \in \cO'\}$ over all subsets $\cO'\subset \cO$.  If $F \subset B_Z(p_\infty,R+\epsilon)$ is compact and $g = \sum \{g_U:~ U \in \cO, U \cap F \ne \emptyset\}$ then $g=1$ on $F$, $0\le g \le 1$ and $g=0$ on the complement of $N_Z^o(F,\epsilon)$ as required.
% and $F \subset B_Z(p_\infty, R) \cap \supp(\mu_\infty)$ be closed. Because $Z$ is proper, $F$ is compact. So there exists $f \in C_c(Z)$ such that $f=1$ on $F$, $f=0$ on $Z \setminus N_Z^o(F,\epsilon)$ and $0\le f \le 1$ everywhere. For all sufficiently large $i$ we have 
%\begin{eqnarray*}
%\mu_i( N_Z^o(F,\epsilon)) &\ge& \int f~d\mu_i > -\epsilon + \int f~d\mu_\infty \ge  -\epsilon + \mu_i(F).
%\end{eqnarray*}

Let $I$ be large enough so that $i>I$ implies $\dist_Z(p_i, p_\infty) < \epsilon$ and $|\mu_i(g) - \mu_\infty(g)| < \epsilon$ for all $g \in \cF$. Let $F \subset B_Z(p_i, R)$ be closed. Then $F \subset B_Z(p_\infty, R+\epsilon)$. So there exists $g \in \cF$ as above. Observe that
\begin{eqnarray*}
\mu_i(F) &\le& \int g~d\mu_i < \epsilon +  \int g~d\mu_\infty \le \epsilon + \mu_\infty(N_Z^o(F,\epsilon)).
\end{eqnarray*}
Similarly, if $F \subset B_Z(p_\infty,R)$ then
\begin{eqnarray*}
\mu_\infty(F) &\le& \int g~d\mu_\infty < \epsilon +  \int g~d\mu_i \le \epsilon + \mu_i(N_Z^o(F,\epsilon)).
\end{eqnarray*}
This shows that $\mu_i, \mu_\infty$ are $(\epsilon,R)$-related. 

Now suppose that for every $\epsilon,R>0$ there exists $I$ such that $i>I$ implies $(\mu_i,p_i)$ and $(\mu_\infty,p_\infty)$ are $(\epsilon,R)$-related. Then there exist sequences $\{\epsilon_i\}_{i=1}^\infty, \{R_i\}_{i=1}^\infty$ such that $\lim_{i \to \infty} \epsilon_i = 0, \lim_{i\to\infty} R_i = +\infty$ and $(\mu_i,p_i)$ and $(\mu_\infty, R_\infty)$ are $(\epsilon_i,R_i)$-related.

\noindent {\bf Claim 1}. For any compact $S \subset Z$, 
$$\lim_{i\to\infty}\mu_i (N_Z^o(S,\epsilon_i)) = \mu_\infty(S).$$
\begin{proof}[Proof of Claim 1]
For all sufficiently large $i$, $S \subset B_Z(p_i, R_i-\epsilon_i) \cap B_Z(p_\infty, R_i-\epsilon_i)$. So \begin{eqnarray*}
\mu_\infty (S ) &\le&  \mu_i (N_Z^o(S,\epsilon_i)) + \epsilon_i \le \mu_\infty (N_Z^o(S,2\epsilon_i) ) + 2\epsilon_i.
\end{eqnarray*}
By taking the limit as $i\to\infty$, the claim follows. This uses the fact that $ \mu_\infty (N_Z^o(S,2\epsilon_i))$ is finite for all sufficiently large $i$ which is true because $\mu_\infty$ is Radon and $Z$ is proper.
\end{proof}

Now let $f$ be a real-valued compactly supported continuous function on $Z$. It suffices to show that $\lim_{i\to\infty} \mu_i(f) = \mu_\infty(f)$. Let $S$ denote the support of $f$ and for $\alpha<\beta$ let 
$$F(\alpha,\beta) = \{x \in Z:~ \alpha\le f(x)\le \beta\} \cap S.$$
Let $\{\alpha_t\}_{t=1}^r$ be a sequence of real numbers such that $\alpha_1 < \min \{f(x):~x\in Z\} < \alpha_2 < \cdots < \max \{f(x):~x\in Z\}<\alpha_r$ and $\mu_\infty(F(\alpha_t,\alpha_t))=0$ for every $t=1\ldots r$.

% We also require that $(_j)_*\vol_{M_j}^{(k)} (\{x \in Z:~f(x)=\alpha_t\}) = 0$ for all $1\le t \le r$ and $j=i,\infty$. 
%Because each measure $\mu_i$ is Radon, 

By Claim 1,
\begin{eqnarray*}
\limsup_{i\to\infty} \int f~d\mu_i &\le& \limsup_{i\to\infty} \sum_{t=1}^{r-1} \alpha_{t+1} \mu_i(N_Z^o(F(\alpha_t,\alpha_{t+1}),\epsilon_i)) \\
&=& \sum_{t=1}^{r-1} \alpha_{t+1} \mu_\infty(F(\alpha_t,\alpha_{t+1})) \le (\sup_{1\le t < r} \alpha_{t+1}-\alpha_t) \mu_\infty(S) + \int f ~d\mu_\infty.
\end{eqnarray*}
We now minimize over all such sequences $\{\alpha_t\}_{t=1}^r$ to obtain $\limsup_{i\to\infty} \int f~d\mu_i  \le \int f ~d\mu_\infty.$ Similarly,
\begin{eqnarray*}
&&\liminf_{i\to\infty} \int f~d\mu_i \\
&\ge& \liminf_{i\to\infty}  \sum_{t=1}^{r-1} \int_{N_Z^o(F(\alpha_t,\alpha_{t+1}),\epsilon_i)} f~d\mu_i -2\|f\|_\infty \sum_{t=1}^r   \mu_i(N_Z^o(F(\alpha_t,\alpha_t),\epsilon_i)) \\
&=&\liminf_{i\to\infty}  \sum_{t=1}^{r-1} \int_{N_Z^o(F(\alpha_t,\alpha_{t+1}),\epsilon_i)} f~d\mu_i \\
 &\ge& \liminf_{i\to\infty} \sum_{t=1}^{r-1} \alpha_{t} \mu_i(N_Z^o(F(\alpha_t,\alpha_{t+1}),\epsilon_i)) =\sum_{t=1}^{r-1} \alpha_{t} \mu_\infty(F(\alpha_t,\alpha_{t+1})) \\
&\ge& -(\sup_{1\le t < r} \alpha_{t+1}-\alpha_t) \mu_\infty(S) + \int f ~d\mu_\infty.
\end{eqnarray*}
Maximizing over all such sequences $\{\alpha_t\}_{t=1}^r$ and combining with the previous inequality, we obtain $\lim_{i\to\infty} \int f~d\mu_i  = \int f ~d\mu_\infty.$ Because $f$ is arbitrary, this implies $\lim_{i\to\infty} \mu_i  = \mu_\infty$ as required.
\end{proof}

%\begin{lem}
%Let $Z$ be a metric space, $p_i \in Z, X_i \subset Z$ for $i=1,2,3$.  If $(X_1,p_1),(X_2,p_2)$ are $(\epsilon_1,R_1)$-related and $(X_2,p_2), (X_3,p_3)$ are $(\epsilon_2,R_2)$-related then $(X_1,p_1),(X_3,p_3)$ are $(\epsilon_1+\epsilon_2, R_3)$-related where $R_3=\min\{R_1-\epsilon_2, R_2-\epsilon_1\}$. 
 %\end{lem}
 
% \begin{proof}
 %This is an exercise.
 
% Let $x \in X_1 \cap B_Z(p_1, R_3)$. 
 
 %There exists $y \in X_2$ such that $\dist_Z(x,y) \le \epsilon_1$.
 
%  So $\dist_Z(y,p_2) \le \dist_Z(p_1,p_2)+\dist_Z(p_1,x) \le \epsilon_1 + R_2-\epsilon_1 \le R_2$. 
  
 % So there exists $z \in X_3$ such that $\dist_Z(y,z) \le \epsilon_2$ which implies $\dist_Z(x,z) \le \epsilon_1+\epsilon_2$. The opposite inequality is similar.
  
  %Also 
 %$$\dist_Z(z,p_3) \le \dist_Z(z,x) + \dist_Z(x,p_1) + \dist_Z(p_1,p_3) \le \epsilon_1+\epsilon_2 + R_3 + \epsilon_1+\epsilon_2.$$
 % \end{proof}

%In this section, we introduce a topology on the set of isomorphism classes of pointed multiply-measured metric spaces. We first need a few preliminary definitions.

\section{Metric measure spaces}

We can now define $(\epsilon,R)$-related pointed mm$^n$-spaces. This will allow us to define open neighborhoods in $\M^n$.

\begin{defn}[$(\epsilon,R)$-related mm$^n$-spaces]\label{defn:Mn-top}
We say that mm$^n$ spaces $(M_1,p_1),(M_2,p_2)$ are $(\epsilon,R)$-related if there exist a metric space $Z$ and isometric embeddings $\varphi_i:M_i \to Z$ such that 
\begin{itemize}
\item $(\varphi_1(M_1),\varphi_1(p_1))$, $(\varphi_2(M_2),\varphi_2(p_2))$ are $(\epsilon,R)$-related as pointed subsets of $Z$; 
\item for every $k=1\ldots n$, $((\varphi_1)_*\vol^{(k)}_{M_1},\varphi_1(p_1))$ and $((\varphi_2)_*\vol^{(k)}_{M_2},\varphi_2(p_2))$ are $(\epsilon,R)$-related as pointed measures of $Z$.
\end{itemize} 
Let $N_{\epsilon,R}(M,p)$ denote the set of all $[M',p'] \in \M^n$ such that $(M',p')$ is $(\epsilon',R')$-related to $(M,p)$ for some $\epsilon'<\epsilon$ and $R' > R$. We show below that this is an open set. %We let $\M^n$ have the topology generated by all such $N_{\epsilon,R}(M,p)$.
%It almost follows from Lemma --, that for any $(M,p) \in \M^n$ that $\{N_{\epsilon,R}(M,p)\}_{\epsilon,R}$ forms a local base for a topology on $\M^n$. We will assume this topology throughout the paper.
\end{defn}

\begin{defn}
A {\em pseudo-metric} $d$ on a set $X$ is a function $d:X \times X \to [0,\infty)$ satisfying all the properties of a metric with one exception: it may happen that $d(x,y)=0$ even if $x\ne y$.
\end{defn}

\begin{lem}\label{lem:pseudo-metric}
Let $Z$ be a set equal to a disjoint union $Z = \sqcup_{i=1}^\infty M_i$ of its subsets $M_i$. Suppose that for each $i$ there is a metric $\dist_{M_i}$ on $M_i$ and there is a collection $\{L_j\}_{j\in J}$ of subsets $L_j \subset Z$ and for each $j$ there is a pseudo-metric $\dist_{L_j}$ on $L_j$. Suppose as well that if $x,y \in L_j \cap M_i$ for some $i,j$ then $\dist_{M_i}(x,y)=\dist_{L_j}(x,y)$. Lastly, we assume that for any $x,y \in Z$ there is a sequence $x=x_1,x_2,\ldots, x_n=y$ such that for each $i$ either $x_i,x_{i+1} \in M_k$ for some $k$ or $x_i, x_{i+1} \in L_j$ for some $j$. Then there is a pseudo-metric $\dist_Z$ on $Z$ such that
\begin{itemize}
\item $\dist_Z(x,y)  = \dist_{M_i}(x,y)$ for any $x,y \in M_i$, for any $i$;
\item $\dist_Z(x,y) \le \dist_{L_j}(x,y)$ for any $x,y \in L_j$ for any $j$.
\end{itemize}
\end{lem}

\begin{proof}
For each $x,y \in Z$ we define $\dist_Z(x,y) = \inf \sum_{k=1}^r \dist_{N_k}(x_k,x_{k+1})$ where the infimum is over all sequences $x=x_1,\ldots, x_r=y$ and choices $N_k \in \{M_i\}_{i=1}^\infty \cup \{L_j\}_{j\in J}$ such that $x_k,x_{k+1} \in N_k$ for all $1\le k < r$. It is easy to check that the conclusions hold.
\end{proof}

\begin{lem}\label{lem:convergent2}
Suppose that $\{(M_i,p_i)\}_{i=1}^\infty$ is a sequence of mm$^n$ spaces such that $(M_i,p_i)$ and $(M_{j},p_{j})$ are $(\epsilon_{ij},R_{ij})$-related for all $i,j$ (where $\epsilon_{ij},R_{ij}$ are positive real numbers). Then there exist a complete separable metric space $Z$ and isometric embeddings $\varphi_i:M_i \to Z$ such that for all $i,j,k$
\begin{itemize}
%\item $\dist_Z(\varphi_i(p_i),\varphi_j(p_j)) < \epsilon_{ij}$ for all $i,j$;
\item $(\varphi_i(M_i),\varphi_i(p_i))$, $(\varphi_j(M_j),\varphi_j(p_j))$ are $(\epsilon_{ij},R_{ij})$-related as pointed subsets of $Z$; 
\item  $((\varphi_i)_*\vol_{M_i}^{(k)},\varphi_i(p_i))$ and  $((\varphi_{j})_*\vol_{M_{j}}^{(k)},\varphi_j(p_j))$ are $(\epsilon_{ij},R_{ij})$-related as pointed measures of $Z$.
%\item the union of the images $\varphi_i(M_i)$ is dense in $Z$.
\end{itemize}
\end{lem}

\begin{proof}
For each $i,j$, there exist a complete separable metric space $Y_{ij}$ and isometric embeddings $\phi_{ij}:M_{i} \to Y_{ij}, \psi_{ij}:M_j \to Y_{ij}$ such that
\begin{itemize}
%\item $\dist_{Y_{ij}}(\phi_{ij}(p_i),\psi_{ij}(p_j)) < \epsilon_{ij}$ for all $i,j$;
\item $(\phi_{ij}(M_i),\phi_{ij}(p_i))$, $(\psi_{ij}(M_j),\psi_{ij}(p_j))$ are $(\epsilon_{ij},R_{ij})$-related as pointed subsets of $Y_{ij}$; 
\item  $((\phi_{ij})_*\vol_{M_i}^{(k)},\phi_{ij}(p_i))$ and  $((\psi_{j})_*\vol_{M_{j}}^{(k)}, \psi_{ij}(p_j))$ are $(\epsilon_{ij},R_{ij})$-related as pointed measures of $Y_{ij}$ for every $k$.
\end{itemize}

Let $Z'$ be the disjoint union of $M_i$ ($i=1,2,\ldots$). By Lemma \ref{lem:pseudo-metric} there exists a pseudo-metric $\dist_{Z'}$ on $Z'$ satisfying:
\begin{itemize}
\item If $x,x' \in M_i \subset Z'$ then $\dist_{Z'}(x,x') = \dist_{M_i}(x,x')$.
\item If $x_i \in M_i, x_j\in M_j$, then $\dist_{Z'}(x_i,x_j) \le  \dist_{Y_{ij}}( \phi_{ij}(x_i), \psi_{ij}(x_{j}))$.
\end{itemize}
This induces an equivalence relation on $Z'$ by: $x \sim y$ if $\dist_{Z'}(x,y)=0$. Let $Z''=Z'/\sim$ with the metric $\dist_{Z''}([x],[y]) = \dist_{Z'}(x,y)$. Let $(Z,\dist_Z)$ be the metric completion of $(Z'',\dist_{Z''})$. For each $i$ there is a canonical isometric embedding $\varphi_i:M_i \to Z$ and the union of the images of these embeddings is dense in $Z$. So $Z$ is separable. %Clearly, $\dist_Z(\varphi_i(p_i),\varphi_j(p_j)) < \epsilon_{ij}$ for all $i,j$.

For any $i,j$, there is a map $\pi_{ij}:\phi_{ij}(M_i)\cup \psi_{ij}(M_j) \to Z$ such that $\pi_{ij}(\phi_{ij}(x_i)) = \varphi_i(x_i)$ if $x_i \in M_i$ and  $\pi_{ij}(\psi_{ij}(x_j)) = \varphi_j(x_j)$ if $x_j \in M_{j}$. This map is distance non-increasing: $\dist_{Y_{ij}}(x,y) \ge \dist_Z(\pi_{ij}(x),\pi_{ij}(y))$. Since $((\phi_{ij})_*\vol_{M_i}^{(k)},\phi_{ij}(p_i))$ and  $((\psi_{j})_*\vol_{M_{j}}^{(k)},\psi_{ij}(p_j))$ are $(\epsilon_{ij},R_{ij})$-related this implies $((\varphi_i)_*\vol_{M_i}^{(k)},\varphi_i(p_i))$ and  $((\varphi_{j})_*\vol_{M_{j}}^{(k)},\varphi_j(p_j))$ are $(\epsilon_{ij},R_{ij})$-related. Similarly, $(\varphi_i(M_i),\varphi_i(p_i))$, $(\varphi_j(M_j),\varphi_j(p_j))$ are $(\epsilon_{ij},R_{ij})$-related as required.
\end{proof}

\begin{lem}\label{lem:eR-related}
If $(M_1,p_1), (M_2,p_2)$ are $(\epsilon_1,R_1)$-related and $(M_2,p_2),(M_3,p_3)$ are $(\epsilon_2,R_2)$-related then $(M_1,p_1),(M_3,p_3)$ are $(\epsilon_1+\epsilon_2,R_3)$-related where $R_3=\min(R_1 - 2\epsilon_2,R_2 -2\epsilon_1)$.
\end{lem}

\begin{proof}
This follows from Lemmas \ref{lem:convergent2} and \ref{lem:related}.
\end{proof}

\begin{prop}\label{prop:convergent}
A sequence $\{[M_i,p_i]\}_{i=1}^\infty \subset \M^n$ converges to  $[M_\infty, p_\infty] \in \M^n$ if and only if for every $\epsilon,R>0$ there exists an $I$ such that $i>I$ implies $(M_i,p_i)$ is $(\epsilon,R)$-related to $(M_\infty,p_\infty)$.
%there is a complete separable proper metric space $Z$ and isometric embeddings $\varphi_i:M_i \to Z$ (for $1\le i \le \infty$) such that 
%\begin{itemize}
%\item $\lim_{i\to\infty} \varphi_i(p_i)=\varphi_\infty(p_\infty)$;
%\item $$\lim_{i \to \infty} (\varphi_i( M_i), \varphi_i(p_i)) = (\varphi_\infty(M_\infty), \varphi_\infty(p_\infty)) $$
%%in the pointed Hausdorff topology;
%\item $\lim_{i\to\infty} (\varphi_i)_*\vol^{(k)}_{M_i} = (\varphi_\infty)_*\vol^{(k)}_{M_\infty}$ (in the weak* topology) for all $k$.
%\end{itemize}
\end{prop}

\begin{proof}
Suppose $\{[M_i,p_i]\}_{i=1}^\infty \subset \M^n$ converges to  $[M_\infty, p_\infty] \in \M^n$. By definition, this means there exist a complete separable proper metric space $Z$ and isometric embeddings $\varphi_i:M_i \to Z$ such that $(\varphi_i( M_i), \varphi_i(p_i))$ converges to $(\varphi_\infty(M_\infty), \varphi_\infty(p_\infty)) $ in the pointed Hausdorff topology and $(\varphi_i)_*\vol_{M_i}^{(k)}$ converges to $(\varphi_\infty)_*\vol_{M_\infty}^{(k)}$ as $i\to\infty$. The proposition now follows from Lemma \ref{lem:key-Mn}.
 
Let us now assume for every $\epsilon,R>0$ there exists $I$ such that $i>I$ imples $(M_i,p_i)$ is $(\epsilon,R)$-related to $(M_\infty,p_\infty)$. By Lemma \ref{lem:eR-related} this implies that for any $i,j >I$, $(M_i,p_i)$ and $(M_j,p_j)$ are $(2\epsilon, R-2\epsilon)$-related. So there exist positive real numbers $\epsilon_i, R_i$ such that 
\begin{itemize}
\item $\lim_{i \to \infty} \epsilon_i = 0, \lim_{i\to\infty} R_i = +\infty$;
\item $(M_i,p_i), (M_j,p_j)$ are $(\epsilon_i,R_i)$-related for every $1\le i<j \le \infty$.
\end{itemize}
By Lemma \ref{lem:convergent2}, there exist a complete separable metric space $Z$ and isometric embeddings $\varphi_i:M_i \to Z$ such that for every $1\le i<j \le \infty$
\begin{itemize}
\item $(\varphi_i(M_i),\varphi_i(p_i)), (\varphi_j(M_j), \varphi_j(p_j))$ are $(\epsilon_i,R_i)$-related;
\item for every $k=1\ldots n$, $((\varphi_i)_*\vol^{(k)}_{M_i}, \varphi_i(p_i)), ((\varphi_j)_*\vol^{(k)}_{M_j}, \varphi_j(p_j))$ are $(\epsilon_i,R_i)$-related;
%\item the union of the images of $\varphi_i(M_i)$ is dense in $Z$.
\end{itemize}
By replacing $Z$ with the closure of the images of the $M_i$'s, we may assume, without loss of generality, that the union $\cup_{i=1}^\infty \varphi_i(M_i)$ is dense in $Z$. Without loss of generality, we may also assume each $M_i \subset Z$ and $\varphi_i$ is the inclusion map. This helps simplify notation.

%Since $\epsilon_i \searrow 0$, this implies $\lim_{i\to\infty} \varphi_i(p_i)=(p_\infty)$. 

We claim that $Z$ is proper. It suffices to show that every ball centered at $p_\infty$ is sequentially compact. So let $R>0$ and $\{x_i\}_{i=1}^\infty \subset B_Z(p_\infty,R)$. There is a sequence $\{y_i\}_{i=1}^\infty$ such that for each $i$, $\dist_Z(x_i,y_i)<1/i$ and $y_i \in M_{n(i)}$ for some $n(i)$. It suffices to show that a subsequence of $\{y_i\}_{i=1}^\infty$ is convergent. If there is some $j$ such that $\{y_i\}_{i=1}^\infty \cap M_j$ is infinite then, since $M_j$ is proper, it follows that there is a convergent subsequence. Otherwise, $\lim_{i\to\infty} n(i) = +\infty$. 

Observe that 
$$\dist_Z(p_i,y_i) \le \dist_Z(p_i,p_\infty) + \dist_Z(p_\infty,x_i)+\dist_Z(x_i,y_i) \le \epsilon_{n(i)} + R + 1/i.$$
In other words, $y_i \in B_Z(p_i,R + 1/i + \epsilon_{n(i)})$. If $i$ is large enough then $R_{n(i)} > R + 1/i + \epsilon_{n(i)}$. Because  $(M_{n(i)},p_{n(i)}), (M_\infty, p_\infty)$ are $(\epsilon_{n(i)},R_{n(i)})$-related, 
$$B_Z(p_i,R + 1/i + \epsilon_{n(i)}) \cap M_{n(i)} \subset N_Z^o( M_\infty, \epsilon_{n(i)}).$$
So there exists $z_i \in M_\infty$ with $\dist_Z(y_i,z_i) \le \epsilon_{n(i)}$. Note
$$\dist_Z(z_i,p_\infty) \le \dist_Z(z_i,y_i) + \dist_Z(y_i,x_i) + \dist_Z(x_i,p_\infty) \le \epsilon_{n(i)} + 1/i + R.$$
Because $M_\infty$ is proper this implies $\{z_i\}_{i=1}^\infty$ has a convergent subsequence. Since $\dist_Z(z_i,x_i) \le \dist_Z(z_i,y_i) + \dist_Z(y_i,x_i) \le \epsilon_{n(i)} + 1/i$ tends to zero as $i\to\infty$, this implies $\{x_i\}_{i=1}^\infty$ has a convergent subsequence as required. 

The proposition now follow from Lemma \ref{lem:key-Mn}.
\end{proof}

%\begin{prop}\label{prop:Mn}
%$\M^n$ is separable and metrizable.
%\end{prop}

%\begin{lem}
%The collection $\{N_{\epsilon,R}(M,p):~ (M,p) \in \M^n, \epsilon,R>0\}$ is a base for the topology of $\M^n$.
%\end{lem}

%\begin{proof}
%It follows from Proposition \ref{prop:convergent} that each $N_{\epsilon,R}(M,p)$ is open and that these sets generate the topology of $\M^n$. 
%\end{proof}

%\begin{defn}
%For $\epsilon, R>0$ and $(M,p) \in \M^n$, let $N_{\epsilon,R}(M,p)$ denote the set of all $(M',p') \in \M^n$ such that $(M',p')$ is $(\epsilon',R')$-related to $(M,p)$ for some $\epsilon'<\epsilon$ and $R' > R$. 
%\end{defn}

\begin{lem}\label{lem:open}
For any $[M,p] \in \M^n$ and $\epsilon,R>0$, the set $N_{\epsilon,R}(M,p)\subset \M^n$ is open.
\end{lem}
\begin{proof}
Let $\{[M_i,p_i]\}_{i=1}^\infty$ be a sequence in $\M^n \setminus N_{\epsilon,R}(M,p)$ which converges to $[M_\infty,p_\infty]$. If $[M_\infty,p_\infty] \in N_{\epsilon,R}(M,p)$ then there is an $\epsilon'<\epsilon$ and $R'>R$ such that $(M_\infty,p_\infty)$ and $(M,p)$ are $(\epsilon',R')$-related. Choose $\epsilon'', R''>0$ so that $\epsilon''+\epsilon'<\epsilon$ and $R < \min(R' - 2\epsilon'', R'' - 2\epsilon')$. By Proposition \ref{prop:convergent}, there is an $i$ such that $(M_i,p_i)$ and $(M_\infty,p_\infty)$ are $(\epsilon'', R'')$-related. Lemma \ref{lem:eR-related} now implies $[M_i,p_i] \in N_{\epsilon,R}(M,p)$. This contradiction proves that the complement of $N_{\epsilon,R}(M,p)$ is closed.
\end{proof}

We can now prove Theorem \ref{thm:Mn} which states $\M^n$ is separable and metrizable.

\begin{proof}[Proof of Theorem \ref{thm:Mn}]
First we show $\M^n$ is metrizable. For $[M,p],[M',p'] \in \M^n$, let 
$$\rho( [M,p], [M',p'] ) = \inf \epsilon + \frac{1}{R+ 2\epsilon}$$
where the infimum is over all $\epsilon,R>0$ such that $[M,p]$ and $[M',p']$ are $(\epsilon,R)$-related. In order to check the triangle inequality, let $[M_i,p_i] \in \M^n$ (for $i=1,2,3$) and suppose $(M_1,p_1),(M_2,p_2)$ are $(\epsilon_1,R_1)$-related and $(M_2,p_2),(M_3,p_3)$ are $(\epsilon_2,R_2)$-related for some $\epsilon_1,\epsilon_2,R_1,R_2>0$. By Lemma \ref{lem:eR-related},
\begin{eqnarray*}
\rho( [M_1,p_1],[M_3,p_3]) &\le& \epsilon_1+\epsilon_2 + \frac{1}{\min\{ R_1 - 2\epsilon_2, R_2 - 2\epsilon_1 \} + 2\epsilon_1+2\epsilon_2} \\
&=& \epsilon_1+\epsilon_2 + \frac{1}{\min\{ R_1 +2\epsilon_1, R_2 + 2\epsilon_2 \} } \\
&\le& \left(\epsilon_1 + \frac{1}{R_1 + 2\epsilon_1}\right) + \left(\epsilon_2 + \frac{1}{R_2 + 2\epsilon_2}\right).
\end{eqnarray*}
By minimizing the right-hand side over all $\epsilon_1,\epsilon_2,R_1,R_2$ such that $(M_1,p_1),(M_2,p_2)$ are $(\epsilon_1,R_1)$-related and $(M_2,p_2),(M_3,p_3)$ are $(\epsilon_2,R_2)$-related, we see that $\rho$ satisfies the triangle inequality. It is therefore a metric on $\M^n$. It is continuous by Lemma \ref{lem:open}. So $\M^n$ is metrizable.

To show that $\M^n$ is separable, let $\F^n_\Q$ be the set of all $[M,p] \in \M^n$ such that $M$ is a finite set, and $\dist_M, \vol^{(1)}_M, \ldots, \vol^{(n)}_M$ are rational-valued. Note $\F^n_\Q$ is countable. We claim that $\F^n_\Q$ is dense in $\M^n$. Let $\F^n$ be the set of all $[M,p] \in \M^n$ such that $M$ is finite. An exercise shows that the closure of $\F^n_\Q$ contains $\F^n$. So it suffices to show that $\F^n$ is dense in $\M^n$. For this purpose, let $[M,p] \in \M^n$. Let $\cM^\F(M)$ denote the set of all measures $\mu \in \cM(M)$ with finite support. It is well-known that $\cM^\F(M)$ is dense in the space of Radon measures on $M$ in the weak* topology. So there exist measures $\mu_i^{(k)} \in \cM^\F(M)$ such that $\lim_{i\to\infty} \mu_i^{(k)} = \vol_M^{(k)}$ for every $k$. Let $X_i$ be a finite subset of $M$ containing $\{p\} \cup \cup_{k=1}^n \supp(\mu_i^{(k)})$ such that $\cup_{i=1}^\infty X_i$ is dense in $M$. We may regard $(X_i,p)$ as an mm$^n$-space with distance $\dist_{X_i}$ equal to the restriction of $\dist_M$ to $X_i$ and measures $\vol_{X_i}^{(k)}$ equal to $\mu_i^{(k)}$. By definition $[X_i,p]$ converges to $[M,p]$ in $\M^n$ as $i\to\infty$. So $\F^n$ and therefore $\F^n_\Q$ is dense in $\M^n$ as claimed.
\end{proof}

\end{document}